\documentclass{amsart}
\usepackage{amssymb,mathrsfs,MnSymbol}
\usepackage{color}

\def\bQ {\mathbf{Q}}
\def\bR {\mathbf{R}}
\def\bS {\mathbf{S}}

\def\bZ {\mathbf{Z}}

\def\cD {\mathcal{D}}

\def\cH {\mathcal{H}}

\def\cM {\mathcal{M}}
\def\cN {\mathcal{N}}

\def\cQ {\mathcal{Q}}

\def\a {{\alpha}}
\def\b {{\beta}}
\def\g {{\gamma}}
\def\Ga {{\Gamma}}
\def\de {{\delta}}
\def\eps {{\epsilon}}
\def\th {{\theta}}
\def\io {{\iota}}
\def\ka {{\kappa}}
\def\l {{\lambda}}
\def\L {{\Lambda}}

\def\d {{\partial}}
\def\grad {{\nabla}}
\def\Dlt {{\Delta}}

\def\rstr {{\big |}}
\def\indc {{\bf 1}}

\def\wto {{\rightharpoonup}}

\newcommand{\Div}{\operatorname{div}}

\newcommand{\Det}{\operatorname{det}}
\newcommand{\Tr}{\operatorname{trace}}

\newcommand{\Esssup}{\mathop{\mathrm{ess\,sup}}}

\def\wto {{\rightharpoonup}}

\newcommand{\ba}{\begin{aligned}}
\newcommand{\ea}{\end{aligned}}

\newcommand{\be}{\begin{equation}}
\newcommand{\ee}{\end{equation}}

\newcommand{\lb}{\label}

\newtheorem{Thm}{Theorem}[section]

\newtheorem{Prop}[Thm]{Proposition}

\newtheorem{Lem}[Thm]{Lemma}
\newtheorem{Def}[Thm]{Definition}

\begin{document}

\title[Partial Regularity for Landau Equation]{Partial Regularity in Time\\ for the Space Homogeneous Landau Equation\\ with Coulomb Potential}

\author[F. Golse]{Fran\c cois Golse}
\address[F.G.]{CMLS, \'Ecole polytechnique, 91128 Palaiseau Cedex, France}
\email{francois.golse@polytechnique.edu}

\author[M. P. Gualdani]{Maria Pia Gualdani}
\address[M.P.G.]{Department of Mathematics, George Washington University, 801 22nd St. NW, Washington DC, 20052, USA}
\email{gualdani@gwu.edu}

\author[C. Imbert]{Cyril Imbert}
\address[C.I.]{CNRS \& DMA, \'Ecole normale sup\'erieure, 75230 Paris Cedex 05, France}
\email{cyril.imbert@ens.fr}

\author[A. Vasseur]{Alexis Vasseur}
\address[A.V.]{Department of Mathematics \& ICES, University of Texas at Austin, 1 University Station C1200, Austin, TX, 78712-0257, USA}
\email{vasseur@math.utexas.edu}

\begin{abstract}
We prove that the set of singular times for weak solutions of the space homogeneous Landau equation with Coulomb potential constructed as in [C. Villani, Arch. Rational Mech. Anal. \textbf{143} (1998), 273--307] has Hausdorff dimension 
at most $\tfrac12$.
\end{abstract}



\date{\today}

\maketitle


\section{Introduction}


We are concerned with the regularity of weak solutions $f\equiv f(t,v)\ge 0$ a.e. to the space homogeneous Landau equation with Coulomb interaction potential
\be\lb{LC}
\d_tf(t,v)=\Div_v\int_{\bR^3}a(v-w)(f(t,w)\grad_vf(t,v)-f(t,v)\grad_wf(t,w))dw\,,\quad v\in\bR^3\,,
\ee
where the collision kernel $a$ is the matrix field
$$
a(z)=\grad^2|z|=\frac{\Pi(z)}{|z|}\,,\quad\text{ with }\Pi(z):=I-\left(\frac{z}{|z|}\right)^{\otimes 2}\,.
$$
(In other words, $\Pi(z)$ is the orthogonal projection on $(\bR z)^\perp$ for all $z\in\bR^3\setminus\{0\}$.) This equation is used in the description of collisions between charged particles in plasma physics (see \cite{Landau36} or \S 41 in 
\cite{LL10}).

Equivalently, the Landau equation with Coulomb potential takes the form
\be\lb{LCncf}
\d_tf(t,v)=\Tr(A[f](t,v)\grad^2_vf(t,v))+8\pi f(t,v)^2
\ee
with $A[f](t,\cdot):=a\star f(t,\cdot)$.

Villani has proved the global existence of a special kind of weak solutions of the Cauchy problem for \eqref{LC}, known as ``H-solutions'', for all initial data with finite mass, energy and entropy (Theorem 4 (i) in \cite{VillaniH}). Whether H-solutions 
of \eqref{LC} with smooth initial data remain smooth for all times or blow up in finite time is one of the outstanding problems in the mathematical analysis of kinetic models: see \S 1.3 (2)  in chapter 5 of Villani's monograph \cite{VillaniHB}. The
form \eqref{LCncf} of the Landau equation suggests that blow-up might occur in finite time, by analogy with the semilinear heat equation $\d_tu(t,x)=\Delta_xu(t,x)+u(t,x)^2$: see the last statement in Theorem 1 of \cite{Weissler80} (for nonnegative 
solutions of the initial boundary value problem on a smooth bounded domain of $\bR^3$ with homogeneous Dirichlet condition at the boundary), or section 5.4 in \cite{Cazenave}. On the other hand, global existence of classical, radially symmetric 
and nonincresing (in the velocity variable) solutions has been established for the equation $\d_tu(t,x)=((-\Dlt_x)^{-1}u)(t,x)\Dlt_xu(t,x)+\a u(t,x)^2$ which can be seen as an ``isotropic'' variant of \eqref{LCncf} (i.e. with the diffusion matrix $A[f]$ replaced with the diffusion coefficient $4\pi((-\Dlt_v)^{-1}f)$ and the constant $8\pi$ replaced with the (smaller) coefficient $4\pi\a$), for all $\a\in[0,\tfrac{74}{74})$ in \cite{GressKriegerStrain,KriegerStrain}, then for $\a=1$ in \cite{GualdaniGuillen}. 

These arguments being somewhat inconclusive, the recent research on the Cauchy problem for \eqref{LC} has produced mostly conditional results --- with one notable exception, which we shall discuss in more detail below. For instance, the 
uniqueness of bounded solutions of \eqref{LC} has been proved in \cite{Fournier}; the regularity of radial $L^p$ solutions with $p>\tfrac32$ has been proved in \cite{GualdaniGuillen} (see also \cite{GualdaniGuillen2}); the case of nonradial $L^p$ 
solutions with $p>\tfrac32$ and moments in $v$ of order larger than $8$ is treated in \cite{Silvestre}, while the large time behavior of H-solutions of \eqref{LC} is discussed in \cite{CarraDesviHe}. Of course, the existence and uniqueness theory
for the space inhomogeneous Landau equation is even harder, and most of the existing results on that equation bear on near-Maxwellian equilibrium global solutions \cite{Guo,CarraMischler,ChenDesviHe}, apart from the very general weak stability 
result in \cite{Lions}. There are also various local existence and uniqueness results, as well as smoothing estimates for solutions of the space inhomogeneous Landau equation under the assumption of locally (in time and space) bounded moments 
in $v$: see \cite{HeYang,HenderSnelsonTarfu,HenderSnelson} (notice that \cite{HenderSnelsonTarfu,HenderSnelson} require only that the distribution function has bounded moments in $v$ of order larger than $9/2$). Since the present paper is 
focused on the Coulomb case, which is the most interesting on physical grounds, we have omitted the rather large literature on the generalizations of the Landau equation where the collision kernel $a$ is replaced with $|z|^{\g+2}\Pi(z)$ with 
$\g>-3$.

Perhaps the most remarkable recent contribution to the mathematical theory of \eqref{LC} is the following result by Desvillettes \cite{DesviJFA}: any (nonnegative) H-solution of \eqref{LC} on $[0,T]\times\bR^3$ with finite mass, energy and entropy 
satisfies the bound
\be\lb{Desvi<}
\int_0^T\int_{\bR^3}\frac{|\grad_v\sqrt{f(t,v)}|}{(1+|v|)^3}dvdt<\infty\,;
\ee
(see Theorem 1 in \cite{DesviJFA}).This bound implies in particular the propagation of moments in $v$ of arbitrary order for such solutions of \eqref{LC} (see Proposition 2 in \cite{DesviJFA}). Both this bound and the propagation of moments are
of key importance in the present work.

Our main result is the following partial regularity statement, which will be presented and discussed in detail in the next section.

\smallskip
\noindent
\textbf{Main Theorem.} \textit{The set of singular times of any H-solution of \eqref{LC} constructed by the approximation scheme described in \cite{VillaniH} has Hausdorff dimension at most $\tfrac12$.}


\section{Main Results}\lb{S-MR}


The prototype of all partial regularity results in partial differential equations is Leray's observation \cite{LerayNS} that the set of singular times of any Leray solution to the Navier-Stokes equations for incompressible fluid dynamics in three space
dimensions has Hausdorff dimension at most $\tfrac12$ (see \S 34 in \cite{LerayNS}, especially, formula (6.5)). Leray's remark was later considerably refined by Scheffer \cite{Scheffer}, and by Caffarelli, Kohn and Nirenberg \cite{CKN1982} 
(see also \cite{Struwe,FLin,VasseurNS})

One key ingredient in Leray's observation is the energy inequality satisfied by Leray solutions to the Navier-Stokes equations. Our first task is therefore to establish an analogous inequality for solutions to \eqref{LC}. Henceforth we denote by
$L^p_k(\bR^3)$ the set of measurable $g\equiv g(v)$ defined a.e. on $\bR^3$ such that
$$
\|g\|_{L^p_k(\bR^3)}:=\left(\int_{\bR^3}(1+|v|^2)^{k/2}|g(v)|^pdv\right)^{1/p}<\infty\,.
$$

We first recall that an H-solution to \eqref{LC} on the time interval $[0,T)$ with initial data $f_{in}\equiv f_{in}(v)\ge 0$ a.e. is an element $f\in C([0,T);\cD'(\bR^3))\cap L^1((0,T);L^1_{-1}(\bR^3))$ such that
\be\lb{ConsLaw}
f(t,v)\ge 0\text{ for a.e. }v\in\bR^3\,,\quad\text{ and }\quad\int_{\bR^3}\left(\begin{matrix}1\\ v\\|v|^2\end{matrix}\right)f(t,v)dv=\int_{\bR^3}\left(\begin{matrix}1\\ v\\|v|^2\end{matrix}\right)f_{in}(t,v)dv
\ee
while
\be\lb{H}
\int_{\bR^3}f(t,v)\ln f(t,v)dv\le\int_{\bR^3}f_{in}(v)\ln f_{in}(v)dv
\ee
for a.e. $t\ge 0$, and
$$
\ba
\int_{\bR^3}f_{in}(v)\phi(0,v)dv+\int_0^T\int_{\bR^3}f(t,v)\d_t\phi(t,v)dv
\\
=\!\!\int_0^T\!\!\int_{\bR^3}\sqrt{\tfrac{f(t,v)f(t,w)}{|v-w|}}(\grad\phi(t,v)\!-\!\grad\phi(t,w))\cdot\Pi(v\!-\!w)(\grad_v\!-\!\grad_w)\sqrt{\tfrac{f(t,v)f(t,w)}{|v-w|}}dvdw
\ea
$$
for each $\phi\in C^1_c([0,T)\times\bR^3)$. Of course, the notion of H-solution is based on the observation that classical solutions of the Landau equation with appropriate decay as $|v|\to+\infty$ satisfy
\be\lb{HThm}
\ba
\frac{d}{dt}H(f)(t)
\\
=-\tfrac12\iint_{\bR^3\times\bR3}\frac{f(t,v)f(t,w)}{|v-w|}\Pi(v-w):(\grad_v\ln f(t,v)-\grad_w\ln f(t,w))^{\otimes 2}dvdw\le 0&\,,
\ea
\ee
where
$$
H(f)(t):=\int_{\bR^3}f(t,v)\ln f(t,v)dv\,.
$$
The notation $H(f)$ to designate this quantity comes from the ``Boltzmann H Theorem'', which is the analogous differential inequality for the Boltzmann equation in the kinetic theory of gases. The positivity of the entropy production 
$-\tfrac{d}{dt}H(f)(t)$ comes from the symmetries in the Landau collision integral (i.e. the right hand side of \eqref{LC}), which are hidden in the nonconservative parabolic form \eqref{LCncf}.

\begin{Def}
A suitable solution of \eqref{LC} on $[0,T)\times\bR^3$ is an H-solution which satisfies, for some negligible set $\cN\subset(0,T)$, some $q\in(1,2)$ and some $C_E>0$, the truncated entropy inequality
\be\lb{SuitSolIneq}
\ba
H_+(f(t_2,\cdot)|\ka)+C'_E\int_{t_1}^{t_2}\left(\int_{\bR^3}|\grad_v(f(t,v)^{1/q}-\ka^{1/q})_+|^qdv\right)^{2/q}dt
\\
\le H_+(f(t_1,\cdot)|\ka)+2\ka\int_{t_1}^{t_2}\int_{\bR^3}(f(t,v)-\ka)_+dvdt
\ea
\ee
for all $t_1<t_2\in[0,T)\setminus\cN$ and all $\ka\ge 1$. For each measurable $g\equiv g(v)\ge 0$ a.e. on $\bR^3$, we denote
$$
H_+(g|\ka):=\int_{\bR^3}\ka h_+\left(\frac{g(v)}{\ka}\right)dv\,,\quad\text{ with }\quad h_+(z):=z(\ln z)_+-(z-1)_+\,.
$$
\end{Def}

The term ``suitable solution'' is used here by analogy with the notion of suitable weak solutions of the Navier-Stokes equations defined in \cite{CKN1982}, which satisfy a local (in space) variant of the Leray energy inequality. Indeed, one can
think of \eqref{SuitSolIneq} as a variant of the entropy inequality for \eqref{LC} localized in the set of values of the solution. 

Our first result is that the approximation scheme used in \cite{VillaniH} to construct an H-solution of the Cauchy problem for \eqref{LC} converges to a suitable solution of \eqref{LC}.

\begin{Prop}\lb{P-ExistSSol}
Let $f_{in}\in L^1(\bR^3)$ satisfy $f_{in}\ge 0$ a.e. on $\bR^3$ and
\be\lb{Momfin}
\int_{\bR^3}(1+|v|^k+|\ln f_{in}(v)|)f_{in}(v)dv<\infty\quad\text{ for some }k>3\,.
\ee
Then, there exists an H-solution of \eqref{LC} with initial data $f_{in}$ and a negligible set $\cN\subset\bR^*_+$ such that $f$ is a suitable solution of \eqref{LC} for each $T\ge 0$ satisfying \eqref{SuitSolIneq} with $q:=\frac{2k}{k+3}$ and 
$C'_E\equiv C'_E[T,q,f_{in}]>0$ for all $t_1<t_2\in[0,T)\setminus\cN$ and all $\ka\ge 1$.
\end{Prop}

\smallskip
The proof of this proposition is based on the analysis in \cite{VillaniH}, with the following additional observations

\noindent
(a) the truncated entropy $-H_+(f|\ka)$ is not a increasing function of time as the original entropy $-H(f)(t)$, but combining the symmetries of the collision integral and the fact that $\Div(\Div a)=\Dlt^2|z|=-4\pi\de_0$ shows that the negative part
of the truncated entropy production involves the depleted nonlinearity $\ka(f-\ka)_+=\min(f,\ka)(f-\ka)_+$ instead of the full nonlinearity $f^2(\ln(f/\ka))_+$ which the nonconservative form \eqref{LCncf} seems to suggest;

\noindent
(b) the Desvillettes argument leading to the inequality \eqref{Desvi<} can be modified to handle the positive part  of the truncated entropy production, and

\noindent
(c) the Desvillettes propagation of moments argument can be used to purge the positive part of the truncated entropy production from the weight $(1+|v|)^{-3}$ at the expense of introducing the exponent $q<2$ in the inequality \eqref{SuitSolIneq}.

The notion of relative entropy of $f$ to a Maxwellian $\cM_{\rho,u,\th}$
$$
-\int_{\bR^3}\left(f(v)\ln\left(\frac{f(v)}{\cM_{\rho,u,\th}(v)}\right)-f(v)+\cM_{\rho,u,\th}(v)\right)dv
$$
with
$$
\cM_{\rho,u,\th}(v):=\frac{\rho}{(2\pi\th)^{3/2}}e^{-|v-u|^2/2\th}
$$
for some $\rho,\th>0$ and $u\in\bR^3$ is used traditionally in kinetic theory to measure the distance of $f$ to $\cM_{\rho,u,\th}$ (see for instance section 3 in \cite{BGL2}). The term $-H_+(f/\ka)$ can be thought of as the truncated variant of
the relative entropy $-H(f|\cM_{(2\pi\th)^{3/2}\ka,0,\th})$ in the large temperature limit $\th\to+\infty$; its time derivative benefits therefore from the same cancellations which lead to the nonnegative entropy production term in \eqref{HThm},
up to the lower order perturbation term due to the truncation, which lead ultimately to the depleted nonlinearity on the right hand side of \eqref{SuitSolIneq}. The truncated entropy $-H(f|\ka)$ is therefore the best imaginable tool for applying
the Stampacchia truncation method to the Landau equation \eqref{LC} without loosing the symmetries of the Landau collision integral. This is precisely the reason why the truncated entropy has been introduced in \cite{GoudonVasseur}
to handle a class of reaction-diffusion systems which is very close to a  kinetic equation (in the discrete velocity setting). In particular, the inequality \eqref{SuitSolIneq} is similar to Lemma 3.1 in \cite{GoudonVasseur} --- notice however the
Remark 3.2 in \cite{GoudonVasseur} which states that extensions of these tools outside of the discrete velocity setting is far from straightforward.

\smallskip
Next we study the occurence of blow-up times for suitable solutions of \eqref{LC}. The set of singular times for a suitable solution of \eqref{LC} is defined by analogy with \S33-34 in \cite{LerayNS}.

\begin{Def}
A real number $\tau>0$ is a regular time for an H-solution $f$ of the Landau equation \eqref{LC} on $[0,+\infty)\times\bR^3$ if there exists $\eps\in(0,\tau)$ such that $f\in L^\infty((\tau-\eps,\tau)\times\bR^3)$. A real number $\tau>0$ is a 
singular time for $f$ if it is not a regular time for $f$. The set of singular times for $f$ in the interval $I\subset(0,+\infty)$ is denoted by $\bS[f,I]$.
\end{Def}

\smallskip
Our main result is the following partial regularity statement on suitable solutions of the Landau equation \eqref{LC}.

\begin{Thm}\lb{T-PartReg}
Let $f$ be a suitable solution to the Landau equation on $[0,T)\times\bR^3$ for all $T>0$ , with initial data $f_{in}$ satisfying the condition
$$
\int_{\bR^3}(1+|v|^k+|\ln f_{in}(v)|)f_{in}(v)dv<\infty\quad\text{ for all }k>3\,.
$$
Then the set $\bS[f,(0,+\infty)]$ of singular positive times for $f$ satisfies
$$
\text{\rm Hausdorff dim }\bS[f,(0,+\infty)]\le\tfrac12\,.
$$
\end{Thm}

\smallskip
Our approach to Theorem \ref{T-PartReg} departs from the general method for proving partial regularity results outlined in section 1 of \cite{CKN1982}. First, we do not use any dimension analysis similar to formula (1.9) of \cite{CKN1982} 
on solutions of the Landau equation. One easily checks that if $f$ is a classical solution of \eqref{LC}, then 
\be\lb{ScalLC}
f_{\l,\mu}(t,v):=\l f(\l t,\mu v)
\ee
is a classical solution of \eqref{LC}. But although the Landau equation \eqref{LC} has a two-parameter family of invariant scaling transformations (richer than the single-para- meter family of invariant scaling transformations of the Navier-Stokes 
equations), the three conserved quantities in \eqref{ConsLaw} and the entropy $-H(f)(t)$ are not simultaneously preserved by any one of these transformations (except for $\l=\mu=1$). Within this family of scaling transformations, we retain
only the case 
\be\lb{ScalLCq}
\l=\mu^\g\quad\text{ with }\quad\g=\frac{5q-6}{2q-2}
\ee
which leaves the inequality \eqref{SuitSolIneq} invariant (upon transforming $\ka$ into $\l\ka$). Notice that this scaling transformation depends on the Lebesgue exponent $q$, which depends itself on the decay in $v$ of $f_{in}$ through
\eqref{Momfin}. The ``optimal'' scaling corresponds to the limit case $k\to\infty$, i.e. to $q\to 2$ and therefore to $\g\to 2$ in \eqref{ScalLCq}. This ``optimal'' scaling is exactly the same as the invariant scaling for the squared velocity field in the 
Navier-Stokes equations, which explains why the bound on the Hausdorff dimension of the set of singular times in both equations is the same. However, the ``optimal'' scaling $\g=2$ is only a limit case and cannot be used directly on \eqref{LC},
at variance with the case of the Navier-Stokes equations.

A first step towards partial regularity is to prove that any suitable solution of the Landau equation \eqref{LC} whose truncated entropy has small enough $L^1$ norm on some time interval is bounded on a smaller time interval.

\begin{Prop}\lb{P-DG} 
Let $f$ be a suitable solution to the Landau equation \eqref{LC} on $[0,1]$ satisfying \eqref{SuitSolIneq} for some negligible set $\cN\subset(0,1]$, some $q\in(\tfrac65,2)$, and some $C'_E>0$. There exists $\eta_0\equiv\eta_0[q,C'_E]>0$ 
such that 
$$
\int_{1/8}^1H_+(f(t,\cdot)|\tfrac12)dt<\eta_0\implies f(t,v)\le 2\quad\text{ for a.e. }(t,v)\in[\tfrac12,1]\times\bR^3\,.
$$
\end{Prop}

This proposition is proved by using the parabolic variant of the De Giorgi method, following the analysis in section 3 of \cite{GoudonVasseur}. However, the proof of this result diverges from the classical De Giorgi method in at least one important
step. While the key idea in De Giorgi's solution \cite{DeGiorgi} of Hilbert's 19\textsuperscript{th} problem is to consider the equations satisfied by the partial derivatives of the extremal as a linear elliptic equations with bounded coefficients, our 
proof of Proposition \ref{P-DG} is based on the truncated entropy inequality which makes critical use of the symmetries of the right hand side of \eqref{LC}. It seems very unclear that Proposition \ref{P-DG} could follow from applying the De Giorgi 
strategy to the conservative form \be\lb{LCcf}
\d_tf(t,v)=\Div_v(A[f](t,v)\grad_vf(t,v)-f(t,v)\Div_vA[f][t,v))
\ee
of the Landau equation without using cancellations suggested by the integral form \eqref{LC} of that equation. For the same reason, the De Giorgi method is used in the present paper in a way that differs from earlier applications of the same
method to other kinetic equations \cite{GIMV,ImbSilv}. Notice that the restriction $q>\tfrac65$ comes from the Sobolev embedding used in the nonlinearization procedure in the De Giorgi method.

\smallskip
The key result leading to partial regularity is the following proposition involving only the dissipation term in \eqref{SuitSolIneq}, by analogy with the argument in \S 34 of \cite{LerayNS}, or Proposition 2 of \cite{CKN1982} for the Navier-Stokes
equations.

\begin{Prop}\lb{P-ImprovedDG}
Let $f$ be a suitable solution to the Landau equation \eqref{LC} on $[0,1]$ satisfying \eqref{SuitSolIneq} for some negligible set $\cN\subset[0,1]$, some $q\in(\tfrac43,2)$, and some $C'_E>0$. There exists $\eta_1\equiv\eta_1[q,C'_E]>0$ and 
$\de_1\in(0,1)$ such that 
$$
\ba
\varlimsup_{\eps\to 0^+}\eps^{\g-3}\int_{1-\eps^\g}^1\left(\int_{\bR^3}|\grad_V(f(T,V)^{1/q}-\eps^{-\g/q})_+|^qdV\right)^{2/q}dT<\eta_1
\\
\implies f\in L^\infty((1-\de_1,1)\times\bR^3)&\,,
\ea
$$
with $\g:=\frac{5q-6}{2q-2}$.
\end{Prop}

The further restriction $q>\tfrac43$ is chose to arrive at the two-level iteration inequality \eqref{IterDGImpr}, which is one of the key ingredients in the proof of Prooposition \ref{P-ImprovedDG}.

Once Proposition \ref{P-ImprovedDG} is proved, the proof of Theorem \ref{T-PartReg} follows by a Vitali covering argument as in section 6 of \cite{CKN1982}. 

The idea of using the De Giorgi method to prove partial regularity in time of suitable solutions of \eqref{LC} comes from \cite{VasseurNS}, where the partial regularity result in \cite{CKN1982} for the Navier-Stokes equations is recast in terms of
the De Giorgi arguments.

The outline of the paper is as follows: the existence proof of suitable solutions (Proposition \ref{P-ExistSSol}) occupies section \ref{S-SSol}. Then, section \ref{S-DG} contains the proof of Proposition \ref{P-DG}, while section \ref{S-ImprovedDG} 
gives the proof of Proposition \ref{P-ImprovedDG}. Finally, the proof of the main theorem (Theorem \ref{T-PartReg}) is given in section \ref{S-PartReg}.


\section{Existence of Suitable Solutions}\lb{S-SSol}


\subsection{Proof of Proposition \ref{P-ExistSSol}} 


The argument is split in several steps, some of which are already used in the construction of a H-solution in \cite{VillaniH}, but need to be recalled for the sake of clarity.

\noindent
\textit{Step 1: Truncated and regularized initial data.} For each integer $n\ge 1$, set
$$
\xi_n(v)=\xi(\tfrac1nv)\,,\quad\zeta_n(v)=n^3\xi(nv)/\|\xi\|_{L^1(\bR^3)}\,,
$$
where 
$$
\xi\in C^\infty(\bR^3)\,,\qquad \indc_{|v|\le 1}\le\xi(v)\le\indc_{|v|\le 2}\,.
$$
Set
$$
f^n_{in}:=\zeta_n\star(\xi_nf_{in})\,,\quad\tilde f^n_{in}(v):=f^n_{in}(v)+\frac1n e^{-|v|^2/2}\,.
$$
Obviously
$$
\int_{\bR^3}f^n_{in}(v)dv=\int_{\bR^3}\xi_n(v)f^n_{in}(v)dv\le\int_{\bR^3}f_{in}(v)dv\,,
$$
and
$$
\ba
\int_{\bR^3}|v|^2f^n_{in}(v)dv\le&\int_{\bR^3}\left(\int_{\bR^3}\zeta_n(w)|v-w|^2dw\right)f_{in}(v)dv
\\
\le&2\int_{\bR^3}\left(\int_{\bR^3}\zeta_n(w)(|v|^2+|w|^2)dw\right)f_{in}(v)dv
\\
\le&\int_{\bR^3}(2|v|^2+O(1/n^2))f_{in}(v)dv\,.
\ea
$$
Hence
$$
\int_{\bR^3}(1+|v|^2)\tilde f^n_{in}(v)dv\le\int_{\bR^3}(1+2|v|^2+O(1/n))f_{in}(v)dv\,.
$$
Likewise since the function $z\mapsto z(\ln z)_+$ is nondecreasing and convex
$$
\zeta_n\star(\xi_nf_{in})(v)(\ln(\zeta_n\star(\xi_nf_{in})(v))_+\le\zeta_n\star(\xi_nf_{in})(v)(\ln(\xi_nf_{in})(v))_+\,,
$$
so that
$$
\ba
\int_{\bR^3}f^n_{in}(v)(\ln f^n_{in}(v))_+dv\le&\int_{\bR^3}(\xi_nf_{in})(v)(\ln(\xi_nf_{in})(v))_+dv
\\
\le&\int_{\bR^3}f_{in}(v)(\ln f_{in}(v))_+dv<\infty\,.
\ea
$$
Because of the elementary inequality\footnote{Indeed 
$$
\ln(a+b)_+-(\ln a)_+\left\{\ba{}&=\ln(1+b/a)\le b/a\le b&&\text{ if }a>1\,,\\ &\le\ln(1+b)\le b&&\text{ if }a\le1\,.\ea\right.
$$}
$$
(\ln(a+b))_+\le(\ln a)_++b\qquad\text{ for all }a,b>0\,,
$$
one has
$$
\ba
\int_{\bR^3}\tilde f^n_{in}(v)(\ln\tilde f^n_{in}(v))_+dv
\\
\le\int_{\bR^3}f^n_{in}(v)(\ln\tilde f^n_{in}(v))_+dv+\int_{\bR^3}\frac1n e^{-|v|^2/2}(\ln\tilde f^n_{in}(v))_+dv
\\
\le\int_{\bR^3}f^n_{in}(v)((\ln f^n_{in}(v))_++\tfrac1n e^{-|v|^2/2})dv+\int_{\bR^3}\tfrac1n e^{-|v|^2/2}\ln(1+\tilde f^n_{in}(v))dv
\\
\le\int_{\bR^3}f^n_{in}(v)((\ln f^n_{in}(v))_++\tfrac2n e^{-|v|^2/2})dv
\\
\le\int_{\bR^3}f^n_{in}(v)((\ln f^n_{in}(v))_++\tfrac2n)dv&\,.
\ea
$$
Hence
$$
\ba
\int_{\bR^3}(1+|v|^2+(\ln\tilde f^n_{in}(v))_+)\tilde f^n_{in}(v)dv
\\
\le\int_{\bR^3}(1+2|v|^2+(\ln f_{in}(v))_++O(1/n))f_{in}(v)dv&\,.
\ea
$$

\noindent
\textit{Step 2: Truncated and regularized Landau equation.} Set
$$
\psi_n(z):=\tfrac1{8\pi}\min\left(\frac1{|z|},n\right)\,,\quad\Pi(z):=\left(I-\frac{z^{\otimes 2}}{|z|^2}\right)
$$
and 
$$
a_n(s):=\psi_n(z)\Pi(z)\,,
$$
so that
$$
\Div a_n(z)=\tfrac1{8\pi}\min\left(\frac1{|z|},n\right)\Div\left(I-\frac{z^{\otimes 2}}{|z|^2}\right)=-\frac{z}{4\pi|z|^3}\indc_{n|z|>1}
$$
and
$$
\Div(\Div a_n)(z)=-\Div\left(\frac{z}{4\pi|z|^3}\indc_{n|z|>1}\right)=\frac1{4\pi|z|^2}\delta(|z|-\tfrac1n)\ge 0\,.
$$

Let $f^n\equiv f^n(t,v)$ be the solution of the truncated and regularized Landau equation
$$
\ba
\d_tf^n(t,v)=&\frac1n\Dlt_vf^n(t,v)
\\
&+\Div_v\int_{\bR^3}\psi_n(|v-w|)\Pi(v-w)(\grad_v-\grad_w)(f^n(t,v)f^n(t,w))dw
\\
f^n\rstr_{t=0}=&\tilde f^n_{in}\,,
\ea
$$
The truncated and regularized Landau has a unique global smooth solution $f^n$ satisfying
$$
f^n(t,v)\ge C_{n,T}e^{-|v|^2/2}>0\,,\quad\text{ for all }v\in\bR^3\text{ and }0\le t\le T\,,
$$
together with 
$$
\ba
{}&\int_{\bR^3}f^n(t,v)dv=\int_{\bR^3}\tilde f^n_{in}(v)dv
\\
&\int_{\bR^3}|v|^2f^n(t,v)dv=\int_{\bR^3}|v|^2\tilde f^n_{in}(v)dv+\tfrac{6}nt\int_{\bR^3}\tilde f^n_{in}(v)dv\,,
\ea
$$
and the following variant of the H Theorem:
\be\lb{HThm-n}
\ba
\int_{\bR^3}f^n\ln f^{n}(t,v)dv+\int_0^tD_{\psi_n}(f^n(s,\cdot))ds+\tfrac4n\int_0^t\int_{\bR^3}\left|\grad_v\sqrt{f^n(s,v)}\right|^2dvds
\\
=\int_{\bR^3}\tilde f^n_{in}\ln\tilde f^n_{in}(v)dv\le\int_{\bR^3}\tilde f^n_{in}(\ln\tilde f^n_{in}(v))_+dv&\,,
\ea
\ee
with the notation
$$
D_{\psi}(g):=2\int_{\bR^6}\psi(|v-w|)\left|\Pi(\grad_v-\grad_w)\sqrt{g(v)g(w)}\right|^2dvdw\,.
$$

Let $G\equiv G(v)$ be the centered, reduced Gaussian distribution
$$
G(v):=\tfrac1{(2\pi)^{3/2}}e^{-|v|^2/2}\,.
$$
Observe that
$$
f\ln f=(f\ln(f/G)-f+G)-(f\ln(1/G)-f+G)\,,
$$
with
$$
(f\ln(f/G)-f+G)\ge 0
$$
and
$$
f\ln(1/G)-f+G)=f(\tfrac12|v|^2+\tfrac32\ln(2\pi)-1)+G\ge 0\,,
$$
so that
$$
f(f\ln f)_+\le(f\ln(f/G)-f+G)\quad\text{ and }\quad f(\ln f)_-\le(f\ln(1/G)-f+G)\,.
$$
Hence \eqref{HThm-n}  becomes
$$
\ba
\int_{\bR^3}f^n(\ln f^{n}(t,v))_+dv+\int_0^tD_{\psi_n}(f^n(s,\cdot))ds+\tfrac4n\int_0^t\int_{\bR^3}\left|\grad_v\sqrt{f^n(s,v)}\right|^2dvds
\\
\le\int_{\bR^3}\tilde f^n_{in}(\ln\tilde f^n_{in}(v))_+dv+\int_{\bR^3}(f^n(t,v)\ln(1/G(v))-f^n(t,v)+G(v))dv
\\
=\int_{\bR^3}\tilde f^n_{in}(\ln\tilde f^n_{in}(v))_+dv+\int_{\bR^3}f^n(t,v)(\tfrac12|v|^2+\tfrac32\ln(2\pi)-1)dv+1
\\
=\int_{\bR^3}\tilde f^n_{in}(\ln\tilde f^n_{in}(v))_+dv+\int_{\bR^3}\tilde f^n_{in}(v)(\tfrac12|v|^2+\tfrac{6t}{n}+\tfrac32\ln(2\pi)-1)dv+1&\,.
\ea
$$
In particular, one has
$$
\ba
\int_{\bR^3}f^n(\ln f^{n}(t,v))_+dv\le&\int_{\bR^3}\tilde f^n_{in}(\ln\tilde f^n_{in}(v))_+dv
\\
&+\int_{\bR^3}\tilde f^n_{in}(v)(\tfrac12|v|^2+\tfrac{6t}{n}+\tfrac32\ln(2\pi)-1)dv+1\,,
\ea
$$
and
$$
\ba
\int_0^tD_{\psi_n}(f^n(s,\cdot))ds+\tfrac4n\int_0^t\int_{\bR^3}\left|\grad_v\sqrt{f^n(s,v)}\right|^2dvds
\\
\le\int_{\bR^3}\tilde f^n_{in}(\ln\tilde f^n_{in}(v))_+dv+\int_{\bR^3}\tilde f^n_{in}(v)(\tfrac12|v|^2+\tfrac{6t}{n}+\tfrac32\ln(2\pi)-1)dv+1&\,.
\ea
$$
In the limit as $n\to\infty$, one has
$$
f^n(t,\cdot)\wto f(t,\cdot)\text{ in }L^1(\bR^3)\text{ uniformly in }t\in[0,T]\text{ for all }T>0\,,
$$
where $f$ is an H-solution of the Cauchy problem for \eqref{LC} with initial data $f_{in}$, according to \cite{VillaniH} on p. 297.

\noindent
\textit{Step 3: A.e. pointwise convergence of $f^n$.} By Theorem 3 in \cite{DesviJFA}, since $\psi_n$ satisfies the assumption used there with $\g_1=-3$ and $K_3=\tfrac1{8\pi}$, one has
$$
\int_{\bR^3}\frac{|\grad_v\sqrt{f^n(t,v)}|^2}{(1+|v|)^3}dv\le C_D[t,\|\tilde f^n_{in}\|_{L^1_2(\bR^3)},\|\tilde f^n_{in}(\ln\tilde f^n_{in})_+\|_{L^1(\bR^3)}](1+D_{\psi_n}(f^n(t,\cdot)))\,.
$$
Hence
$$
\ba
\frac1{C_D}\int_0^t\int_{\bR^3}\frac{|\grad_v\sqrt{f^n(s,v)}|^2}{(1+|v|)^3}dvds+\tfrac4n\int_0^t\int_{\bR^3}\left|\grad_v\sqrt{f^n(s,v)}\right|^2dvds
\\
\le\int_{\bR^3}\tilde f^n_{in}(\ln\tilde f^n_{in}(v))_+dv+\int_{\bR^3}\tilde f^n_{in}(v)(\tfrac12|v|^2+\tfrac{6t}{n}+\tfrac32\ln(2\pi)-1)dv+1+t&\,.
\ea
$$

Next we apply Proposition 4 in \cite{DesviJFA}. Observe that 
$$
\psi_n(z)=\min\left(\frac1{|z|},n\right)
$$
satisfies the assumption in Proposition 4 of \cite{DesviJFA} with $\g_1=\g_2=-3$ and $\de=2$, and with $K_1=K_2=K_3=1$. By Lemma 3 of \cite{DesviJFA}
$$
\int_{\bR^3}\frac{f^n(t,v)^3dv}{(1+|v|)^9}\le C\left(\int_{\bR^3}f^n(t,v)dv+\int_{\bR^3}\frac{|\grad_v\sqrt{f^n(t,v)}|^2}{(1+|v|)^3}dv\right)\,,
$$
so that $\cQ_{T,3,-3}(f^n)\le C_T$ for all $T>0$. Then, for all $T>0$ and all $t\in[0,T]$, one has
$$
\int_{\bR^3}(1+|v|^2)^kf(t,v)dv\le C'_D[T,\|\tilde f^n_{in}\|_{L^1_{2k}(\bR^3)},\|\tilde f^n_{in}(\ln\tilde f^n_{in})_+\|_{L^1(\bR^3)},\cQ_{T,3,-3}(f^n)]\,.
$$

By H\"older's inequality, choosing $p'=2/q$ (so that $q\in(1,2)$)
$$
\ba
\int_{\bR^3}|\grad_vf^n(t,v)^{1/q}|^qdv=(\tfrac2q)^q\int_{\bR^3}f^n(t,v)^{1/p}|\grad_v\sqrt{f^n(t,v)}|^{2/p'}dv
\\
=(\tfrac2q)^q\int_{\bR^3}(1+|v|^2)^{3/2p'}f^n(t,v)^{1/p}(1+|v|^2)^{-3/2p'}|\grad_v\sqrt{f^n(t,v)}|^{2/p'}dv
\\
\le(\tfrac2q)^q\left(\int_{\bR^3}(1+|v|^2)^{3p/2p'}f^n(t,v)dv\right)^{1/p}\left(\int_{\bR^3}\frac{|\grad_v\sqrt{f^n(t,v)}|^2}{(1+|v|^2)^{3/2}}dv\right)^{1/p'}
\\
\le (\tfrac2q)^qC'_D[T,\|\tilde f^n_{in}\|_{L^1_{3p-3}(\bR^3)},\|\tilde f^n_{in}(\ln\tilde f^n_{in})_+\|_{L^1(\bR^3)},\cQ_{T,3,-3}(f^n)]^{1/p}
\\
\times\left(\int_{\bR^3}\frac{|\grad_v\sqrt{f^n(t,v)}|^2}{(1+|v|^2)^{3/2}}dv\right)^{1/p'}&\,.
\ea
$$
Hence
$$
\ba
\left(\int_{\bR^3}|\grad_vf^n(t,v)^{1/q}|^qdv\right)^{2/q}
\\
\le(\tfrac2q)^2C'_D[T,\|\tilde f^n_{in}\|_{L^1_{3p-3}(\bR^3)},\|\tilde f^n_{in}(\ln\tilde f^n_{in})_+\|_{L^1(\bR^3)},\cQ_{T,3,-3}(f^n)]^{\frac1{p-1}}
\\
\times\int_{\bR^3}\frac{|\grad_v\sqrt{f^n(t,v)}|^2}{(1+|v|^2)^{3/2}}dv
\\
\le(\tfrac2q)^2C'_D[T,\|\tilde f^n_{in}\|_{L^1_{3p-3}(\bR^3)},\|\tilde f^n_{in}(\ln\tilde f^n_{in})_+\|_{L^1(\bR^3)},\cQ_{T,3,-3}(f^n)]^{\frac1{p-1}}
\\
\times C_D[T,\|\tilde f^n_{in}\|_{L^1_2(\bR^3)},\|\tilde f^n_{in}(\ln\tilde f^n_{in})_+\|_{L^1(\bR^3)}](1+D_{\psi_n}(f^n(t,\cdot))&\,.
\ea
$$

Therefore
$$
\ba
C_E\int_0^t\left(\int_{\bR^3}|\grad_vf^n(s,v)^{1/q}|^qdv\right)^{2/q}ds+\tfrac4n\int_0^t\int_{\bR^3}\left|\grad_v\sqrt{f^n(s,v)}\right|^2dvds
\\
\le\int_{\bR^3}\tilde f^n_{in}(\ln\tilde f^n_{in}(v))_+dv+\int_{\bR^3}\tilde f^n_{in}(v)(\tfrac12|v|^2+\tfrac{6t}{n}+\tfrac32\ln(2\pi)-1)dv+1+t&\,,
\ea
$$
with
$$
\ba
\frac1{C_E}:=(\tfrac2q)^2C'_D[T,\|\tilde f^n_{in}\|_{L^1_{3p-3}(\bR^3)},\|\tilde f^n_{in}(\ln\tilde f^n_{in})_+\|_{L^1(\bR^3)},\cQ_{T,3,-3}(f^n)]^{\frac1{p-1}}
\\
\times C_D[T,\|\tilde f^n_{in}\|_{L^1_2(\bR^3)},\|\tilde f^n_{in}(\ln\tilde f^n_{in})_+\|_{L^1(\bR^3)}]\,.
\ea
$$

Hence
$$
(f^n)^{1/q}=O(1)_{L^\infty(0,T;L^q(\bR^3))}\quad\text{ and }\quad\grad_v(f^n)^{1/q}=O(1)_{L^2(0,T;L^q(\bR^3))}\,.
$$
In particular, 
$$
\grad_vf^n=q(f^n)^{1/q'}\grad_v(f^n)^{1/q}=O(1)_{L^2(0,T;L^1(\bR^3))}\,,
$$
since 
$$
f^n=O(1)_{L^\infty(\bR_+;L^1(\bR^3))}\,.
$$

On the other hand
$$
\ba
\d_t\int_{\bR^3}f^n(t,v)\phi(v)dv=\frac1n\int_{\bR^3}\Dlt\phi(v)f^n(t,v)dv
\\
-\iint_{\bR^6}\psi_n(v-w)\sqrt{f^n(t,v)f^n(t,w)}\Pi(v-w)(\grad_v-\grad_w)\sqrt{f^n(t,v)f^n(t,w)}
\\
\cdot(\grad\phi(v)-\grad\phi(w))dvdw
\ea
$$
and
$$
\ba
\Big|\iint_{\bR^6}\psi_n(v-w)\sqrt{f^n(t,v)f^n(t,w)}\Pi(v-w)(\grad_v-\grad_w)\sqrt{f^n(t,v)f^n(t,w)}
\\
\cdot(\grad\phi(v)-\grad\phi(w))dvdw\Big|\le\|\grad^2\phi\|_{L^\infty}\|f^{n}(t,\cdot)\|_{L^1(\bR^3)}
\\
\times\left(\iint_{\bR^6}\psi_n(v-w)\left|\Pi(v-w)(\grad_v-\grad_w)\sqrt{f^n(t,v)f^n(t,w)}\right|^2dvdw\right)^{1/2}&\,,
\ea
$$
for all $\phi\in W^{2,\infty}(\bR^3)$, so that
$$
\ba
\left\|\d_t\int_{\bR^3}f^n(t,v)\phi(v)dv\right\|_{L^2(0,T)}
\\
\le C[T,\|\tilde f^n_{in}\|_{L^1_2(\bR^3)},\|\tilde f^n_{in}(\ln\tilde f^n_{in})_+\|_{L^1(\bR^3)}]\|\grad^2\phi\|_{L^\infty(\bR^3)}&\,.
\ea
$$
In other words
$$
f^n=O(1)_{L^2(0,T;W^{-2,1}(\bR^3))}\,.
$$
By the Aubin-Lions lemma (Lemmas 24.5 and 24.3 in \cite{TartarNSE}), $f^n$ is relatively compact in $L^1_{loc}(\bR_+\times\bR^3)$ and, possibly after extracting a subsequence of $f^n$, one has 
$$
f^n\to f\text{ a.e. on }[0,+\infty)\times\bR^3\,.
$$

\smallskip
\noindent
\textit{Step 4: Truncated entropy inequality} For $\ka>0$, multiplying both sides of the truncated and regularized Landau equation by $(\ln(f^n(t,v)/\ka))_+$ and integrating in $v$ shows that
$$
\ba
H_+(f^n(t_2,\cdot)|\ka)+\frac1n\int_{t_1}^{t_2}\int_{\bR^3}\frac{|\grad_v(f^n(t,v)-\ka)_+|^2}{f^n(t,v)}dv
\\
+\tfrac12\int_{t_1}^{t_2}\iint_{\bR^6}a_n(v-w):\left(\grad_v\left(\ln\frac{f^n(t,v)}{\ka}\right)_+\!\!\!-\grad_w\left(\ln\frac{f^n(t,w)}{\ka}\right)_+\right)^{\otimes 2}
\\
\times f^n(t,v)f^n(t,w)dvdwdt
\\
=-\int_{t_1}^{t_2}\iint_{\bR^6}a_n(v-w):\grad_v\left(\ln\frac{f^n(t,v)}{\ka}\right)_+\otimes\grad_w\left(\ln\frac{f^n(t,w)}{\ka}\right)_-
\\
\times f^n(t,v)f^n(t,w)dvdwdt+H_+(f^n(t_1,\cdot)|\ka)&\,.
\ea
$$
Now
$$
\ba
-\iint_{\bR^6}a_n(v-w)f^n(t,v)f^n(t,w):\grad_v\left(\ln\frac{f^n(t,v)}{\ka}\right)_+\otimes\grad_w\left(\ln\frac{f^n(t,w)}{\ka}\right)_-dvdw
\\
=-\iint_{\bR^6}a_n(v-w):\grad_v(f^n(t,v)-\ka)_+\otimes\grad_w((\ka-f^n(t,w))_+-\ka)dvdw
\\
=\iint_{\bR^6}\Div(\Div a_n)(v-w)(f^n(t,v)-\ka)_+(\ka-(\ka-f^n(t,w))_+)dvdw&\,,
\ea
$$
so that
$$
\ba
-\iint_{\bR^6}a_n(v-w)f^n(t,v)f^n(t,w):\grad_v\left(\ln\frac{f^n(t,v)}{\ka}\right)_+\otimes\grad_w\left(\ln\frac{f^n(t,w)}{\ka}\right)_-dvdw
\\
=\iint_{\bR^6}\frac1{4\pi|v-w|^2}\delta(|v-w|-\eps)(f^n(t,v)-\ka)_+(\ka-(\ka-f^n(t,w))_+)dvdw
\\
\le\frac{\ka}{4\pi}\iint_{\bR^6}\frac1{|v-w|^2}\delta(|v-w|-\eps)(f^n(t,v)-\ka)_+dvdw
\\
=\ka\int_{\bR^3}(f^n(t,v)-\ka)_+dvdw&\,.
\ea
$$
In the end, for each $\ka>0$ and each $t_1,t_2$ such that $0\le t_1\le t_2<\infty$, one has
\be\lb{H+nIneq}
\ba
H_+(f^n(t_2,\cdot)|\ka)+\frac1n\int_{t_1}^{t_2}\int_{\bR^3}\frac{|\grad_v(f^n(s,v)-\ka)_+|^2}{f^n(s,v)}dvds
\\
+\tfrac12\int_{t_1}^{t_2}\!\!\!\iint_{\bR^6}\!\!b_n(v\!-\!w):\left(\grad_v\left(\ln\frac{f^n(s,v)}{\ka}\right)_+\!\!\!-\grad_w\left(\ln\frac{f^n(s,w)}{\ka}\right)_+\right)^{\otimes 2}\!\!
\\
\times f^n(s,v)f^n(s,w)dvdwds
\\
\le\ka\int_{t_1}^{t_2}\int_{\bR^3}(f^n(s,v)-\ka)_+dvds+H_+(f^n(t_1,\cdot)|\ka)&\,.
\ea
\ee

\smallskip
\noindent
\textit{Step 5: A local lower bound for the truncated entropy production} We begin with the following auxiliary result, whose proof is deferred until the end of the present section. This is an extension of Theorem 3 in \cite{DesviJFA} to truncated
entropies.

\begin{Lem}\lb{L-CoercH+}
The sequence $f^n$ constructed in step 1 of the present section satisfies the inequality:
$$
\ba
\int_{\bR^3}\frac{|\grad_v\sqrt{f^n(t,v)}|^2}{(1+|v|)^3}\indc_{f^n(t,v)>\ka}dv
\\
\le C''_D\iint_{\bR^6}\tfrac{f^n(t,v)f^n(t,w)}{|v-w|^3}\Pi(v-w):\left(\grad_v\left(\ln\tfrac{f^n(t,v)}\ka\right)_+-\grad_w\left(\ln\tfrac{f^n(t,w)}\ka\right)_+\right)^{\otimes 2}dvdw
\\
+C''_D\int_{\bR^3}(f^n(t,w)-\ka)_+dw&\,,
\ea
$$
where
$$
\ba
C''_D\equiv C''_D[M_0(f^n)M_2(t,f^n)^2,H(f_{in})]
\\
:=\!18C[M_2(t,f^n)^2,H(f_{in})]^2M_2(t,f^n)^4M_0(f^n)\left(1\!\!+\!\!\tfrac1{\sqrt{e\l}}\right)^2\max\left(4,4\l^2,(\tfrac3{2e\l})^{3/2}\right)&,
\ea
$$
with
$$
M_0(f^n):=\int_{\bR^3}f^n(t,v)dv\,,\quad M_2(t,f^n):=\int_{\bR^3}(1+|v|^2)f^n(t,v)dv\,,
$$
and where $C''[M_2(t,f^n)^2,H(f_{in})]$ is the constant $C(N,\bar H)$ of Lemma 2 of \cite{DesviJFA}.
\end{Lem}


\smallskip
Applying Lemma \ref{L-CoercH+} shows that
$$
\ba
H_+(f^n(t_2,\cdot)|\ka)+&\frac1n\int_{t_1}^{t_2}\int_{\bR^3}\frac{|\grad_v(f^n(s,v)-\ka)_+|^2}{f^n(s,v)}dvds
\\
+&\frac1{C''_D}\int_{t_1}^{t_2}\int_{\bR^3}\frac{|\grad_v\sqrt{f^n(t,v)}|^2}{(1+|v|)^3}\indc_{f^n(s,v)\ge\ka}dvds
\\
\le&(1+\ka)\int_{t_1}^{t_2}\int_{\bR^3}(f^n(s,v)-\ka)_+dvds+H_+(f^n(t_1,\cdot)|\ka)\,,
\ea
$$
for all $t_1,t_2$ such that $0\le t_1<t_2<\infty$. By Proposition 4 \cite{DesviJFA} recalled in step 3 above, for all $T>0$ and all $t\in[0,T]$, one has
$$
\int_{\bR^3}(1+|v|^2)^kf(t,v)dv\le C'_D[T,\|\tilde f^n_{in}\|_{L^1_{2k}(\bR^3)},\|\tilde f^n_{in}(\ln\tilde f^n_{in})_+\|_{L^1(\bR^3)},\cQ_{T,3,-3}(f^n)]\,.
$$

By H\"older's inequality, choosing $p'=2/q$ (so that $q\in(1,2)$)
$$
\ba
\int_{\bR^3}|\grad_v(f^n(t,v)^{1/q}-\ka^{1/q})_+|^qdv=\int_{\bR^3}|\grad_vf^n(t,v)^{1/q}|^q\indc_{f^n(t,v)\ge\ka}dv
\\
=(\tfrac2q)^q\int_{\bR^3}f^n(t,v)^{1/p}|\grad_v\sqrt{f^n(t,v)}|^{2/p'}\indc_{f^n(t,v)\ge\ka}dv
\\
=(\tfrac2q)^q\int_{\bR^3}(1+|v|^2)^{3/2p'}f^n(t,v)^{1/p}(1+|v|^2)^{-3/2p'}|\grad_v\sqrt{f^n(t,v)}|^{2/p'}\indc_{f^n(t,v)\ge\ka}dv
\\
\le(\tfrac2q)^q\left(\int_{\bR^3}(1+|v|^2)^{3p/2p'}f^n(t,v)dv\right)^{1/p}\left(\int_{\bR^3}\frac{|\grad_v\sqrt{f^n(t,v)}|^2\indc_{f^n(t,v)\ge\ka}}{(1+|v|^2)^{3/2}}dv\right)^{1/p'}
\\
\le (\tfrac2q)^qC'_D[T,\|\tilde f^n_{in}\|_{L^1_{3p-3}(\bR^3)},\|\tilde f^n_{in}(\ln\tilde f^n_{in})_+\|_{L^1(\bR^3)},\cQ_{T,3,-3}(f^n)]^{1/p}
\\
\times\left(\int_{\bR^3}\frac{|\grad_v\sqrt{f^n(t,v)}|^2\indc_{f^n(t,v)\ge\ka}}{(1+|v|^2)^{3/2}}dv\right)^{1/p'}\,.
\ea
$$

Hence
$$
\ba
H_+(f^n(t_2,\cdot)|\ka)+&\frac1n\int_{t_1}^{t_2}\int_{\bR^3}\frac{|\grad_v(f^n(s,v)-\ka)_+|^2}{f^n(s,v)}dvds
\\
+&C'_E\int_{t_1}^{t_2}\left(\int_{\bR^3}|\grad_v(f^n(t,v)^{1/q}-\ka^{1/q})_+|^qdv\right)^{2/q}ds
\\
\le&(1+\ka)\int_{t_1}^{t_2}\int_{\bR^3}(f^n(s,v)-\ka)_+dvds+H_+(f^n(t_1,\cdot)|\ka)\,,
\ea
$$
with 
$$
\ba
\frac1{C'_E}:=(\tfrac2q)^qC'_D[T,\|\tilde f^n_{in}\|_{L^1_{3p-3}(\bR^3)},\|\tilde f^n_{in}(\ln\tilde f^n_{in})_+\|_{L^1(\bR^3)},\cQ_{T,3,-3}(f^n)]^{1/p}
\\
\times C''_D[M_0(f^n)M_2(T,f^n)^2,H(f_{in})]&\,.
\ea
$$

\smallskip
\noindent
\textit{Step 6: Passing to the limit in the truncated entropy.} By Sobolev embedding, for each $T>0$, one has
$$
(f^n)^{1/q}=O(1)_{L^2(0,T;L^{q^*}(\bR^3))}\,,
$$
and hence
$$
f^n=O(1)_{L^{2/q}(0,T;L^{q^*/q}(\bR^3))}\,.
$$
Since $1<q<2<q'$,  and
$$
(1+|v|^2)f^n=O(1)_{L^\infty(0,T;L^1(\bR^3))}\,,
$$
this implies that
$$
(1+|v|^2)^{\frac{2-q}{4-q}}f^n=O(1)_{L^{\frac4q-1}([0,T]\times\bR^3)}
$$
(applying Proposition 5 of \cite{DesviJFA} with $1-\b=\tfrac{2}{4-q}$). 

Then
$$
\ba
\int_0^T\left|\int_{\bR^3}(h_+(f(t,v)/\ka)-h_+(f^n(t,v)/\ka))dv\right|dt
\\
\le\int_0^T\int_{\bR^3}|h_+(f(t,v)/\ka)-h_+(f^n(t,v)/\ka)|dvdt&\to 0\text{ as }n\to\infty\,.
\ea
$$
Indeed, $h_+(f(t,v)/\ka)-h_+(f^n(t,v)/\ka)\to 0$ for a.e. $(t,v)\in[0,T)\times\bR^3$ as $n\to\infty$; moreover, the sequence $h_+(f/\ka)-h_+(f^n/\ka)$ is equiintegrable and tight on $[0,T]\times\bR^3$ since we have seen that there exists 
$\a_1,\a_2>0$ such that
$$
(1+|v|^2)^{\a_1}(f^n)^{1+\a_2}=O(1)_{L^1([0,T]\times\bR^3)}\,.
$$
Hence, for all $\ka>0$, one has
$$
H_+(f^n|\ka)\to H_+(f|\ka)\text{ in }L^1_{loc}([0,+\infty))\text{ as }n\to+\infty\,,
$$
and hence, possibly after extracting a subsequence if necessary
$$
H_+(f^n(t,\cdot)|\ka)\to H_+(f(t,\cdot)|\ka)\text{ for a.e. }t\ge 0\text{ as }n\to+\infty\,.
$$

The subsequence and the negligible exceptional set of times where the pointwise convergence above is not valid may depend on $\kappa$. Henceforth, we restrict our attention to the case of $\kappa$ being a rational number, and since 
there are countably many such numbers, the union of the exceptional sets corresponding to each rational $\kappa$ is another negligible set $\cN$. Hence, by diagonal extraction
$$
H_+(f^n(t,\cdot)|\ka)\to H_+(f(t,\cdot)|\ka)\text{ for all }t\in\bR_+\setminus\cN\text{ and all }\ka\in\bQ\cap[1,+\infty)
$$
as $n\to\infty$.

\smallskip
\noindent
\textit{Step 7: Passing to the limit in the truncated entropy production.} One has
$$
((f^n)^{1/q}-\ka^{1/q})_+\to(f^{1/q}-\ka^{1/q})_+\text{ a.e. on }\bR_+\times\bR^3\,,
$$
and
$$
((f^n)^{1/q}-\ka^{1/q})_+^q\le f^n=O(1)_{L^\infty(0,T;L^1_{-2}(\bR^3))}\,.
$$
Hence $((f^n)^{1/q}-\ka^{1/q})_+$ is equiintegrable and tight on $[0,T]\times\bR^3$, so that
$$
((f^n)^{1/q}-\ka^{1/q})_+\to(f^{1/q}-\ka^{1/q})_+\text{ in }L^1([0,T]\times\bR^3)\,.
$$
By continuity of the gradient in the sense of distributions, we already know that
$$
\grad_v((f^n)^{1/q}-\ka^{1/q})_+\to\grad_v(f^{1/q}-\ka^{1/q})_+\text{ in }\cD'((0,T)\times\bR^3)\,.
$$

On the other hand
$$
\grad_v((f^n)^{1/q}-\ka^{1/q})_+=\indc_{f^n\ge\ka}\grad_v(f^n)^{1/q}=O(1)_{L^2(0,T;L^q(\bR^3))}\,.
$$
Since $L^2(0,T;L^q(\bR^3))$ is reflexive because $1<q<2$ (see Corollary 2 in Chapter IV, \S 1 of \cite{DiestelUhl}), the Banach-Alaloglu theorem implies that, possibly after extracting a subsequence,
$$
\grad_v((f^n)^{1/q}-\ka^{1/q})_+\wto\ell\quad\text{ in }L^2(0,T;L^q(\bR^3))\,,
$$
and therefore in $\cD'((0,T)\times\bR^3)$. Hence 
$$
\ell=\grad_v(f^{1/q}-\ka^{1/q})_+\,,
$$
and by compactness and uniqueness of the limit
$$
\grad_v((f^n)^{1/q}-\ka^{1/q})_+\wto\grad_v(f^{1/q}-\ka^{1/q})_+\quad\text{ in }L^2(0,T;L^q(\bR^3))\,.
$$

For all $t_1<t_2\in\bR_+\setminus\cN$, and all $\ka\in\bQ\cap[1,+\infty)$, one has
$$
\ba
\int_{t_1}^{t_2}\left(\int_{\bR^3}|\grad_v((f^n)^{1/q}-\ka^{1/q})_+|^qdv\right)^{2/q}dt\le &H_+(f^n(t_1,\cdot)|\ka)-H_+(f^n(t_2,\cdot)|\ka)
\\
&+(1+\ka)\int_{t_1}^{t_2}\int_{\bR^3}(f^n(t,v)-\ka)_+dvdt\,,
\ea
$$
and passing to the limit as $n\to+\infty$, one finds that
$$
\ba
\int_{t_1}^{t_2}\left(\int_{\bR^3}|\grad_v(f^{1/q}-\ka^{1/q})_+|^qdv\right)^{2/q}dt\le&\varliminf_{n\to\infty}\int_{t_1}^{t_2}\left(\int_{\bR^3}\grad_v((f^n)^{1/q}-\ka^{1/q})_+^qdv\right)^{2/q}\!\!dt
\\
\le&H_+(f(t_1,\cdot)|\ka)-H_+(f(t_2,\cdot)|\ka)
\\
&+(1+\ka)\int_{t_1}^{t_2}\int_{\bR^3}(f(t,v)-\ka)_+dvdt\,,
\ea
$$
where the first inequality is Proposition III.5 (iii) of \cite{Brezis}.

\smallskip
\noindent
\textit{Step 8: Writing \eqref{SuitSolIneq} for all $\ka\ge 1$.} Let $\ka>1$, and let $\ka_j$ be an increasing sequence of rational numbers converging to $\ka$. Since $h_+$ is nondecreasing, for all $t\in\bR_+\setminus\cN$, one has 
$$
0\le h_+(f(t,v)/\ka_j)\le h_+(f(t,v)/\ka_1)\quad\text{ for a.e. }v\ni\bR^3\,,
$$
so that 
$$
\frac1{\ka_j}H_+(f(t,\cdot)|\ka_j)\to\frac1{\ka}H_+(f(t,\cdot)|\ka)\quad\text{ as }j\to+\infty
$$
for all $t\in\bR_+\setminus\cN$ by dominated convergence. By the same token
$$
|\grad_v(f^{1/q}-\ka_j^{1/q})_+|^q=|\grad_v(f^{1/q})|^q\indc_{f\ge\ka_j}\le|\grad_v(f^{1/q})|^q\indc_{f\ge\ka_1}=|\grad_v(f^{1/q}-\ka_1^{1/q})_+|^q\,,
$$
a.e. on $(0,+\infty)\times\bR^3$, so that
$$
\int_{t_1}^{t_2}\int_{\bR^3}|\grad_v(f^{1/q}-\ka_j^{1/q})_+|^qdvdt\to\int_{t_1}^{t_2}\int_{\bR^3}|\grad_v(f^{1/q}-\ka_j^{1/q})_+|^qdvdt
$$
as $j\to+\infty$ by dominated convergence. Since $f$ satisfies \eqref{SuitSolIneq} for each $k_\in\bQ\cap[1,+\infty)$, and since this set is dense in $[1,+\infty)$, we conclude from the argument above  that $f$ satisfies \eqref{SuitSolIneq} 
for each $\ka\ge 1$, by passing to the limit as $j\to+\infty$ in \eqref{SuitSolIneq} written for $\ka_j$.

This concludes the proof of Proposition \ref{P-ExistSSol}

\subsection{Proof of Lemma \ref{L-CoercH+}} 


This proof is a simple variant of the proof of Theorem 3 in \cite{DesviJFA}, with a few additional terms which require a specific treatment. 

Start from the elementary identity
$$
\Tr((X\wedge Y)^2)=2(X\cdot Y)^2-2|X|^2|Y|^2\,,
$$
which holds for all $X,Y\in\bR^3$, with the notation $X\wedge Y=X\otimes Y-Y\otimes X$. For all $v\not=w\in\bR^3$, setting $X=v-w$ and
$$
Y:=\grad_v\left(\ln\frac{f^n(t,v)}{\ka}\right)_+-\grad_w\left(\ln\frac{f^n(t,w)}{\ka}\right)_+
$$
leads to the identity
\be\lb{Ident-qij}
\ba
(I-\Pi(v-w)):\left(\grad_v\left(\ln\frac{f^n(t,v)}{\ka}\right)_+-\grad_w\left(\ln\frac{f^n(t,w)}{\ka}\right)_+\right)^{\otimes 2}
\\
=\sum_{1\le i<j\le 3}\frac{|q_{ij}^{f^n}(t,v,w)|^2}{|v-w|^2}&\,,
\ea
\ee
where
$$
\ba
q_{ij}^{f^n}(t,v,w)&=\left|\begin{matrix}v_i-w_i &Z_i-\d_{w_i}\left(\ln\frac{f^n(t,w)}{\ka}\right)_+\\  \\ v_j-w_j&Z_j-\d_{w_j}\left(\ln\frac{f^n(t,w)}{\ka}\right)_+\end{matrix}\right|
\\
&=Z_{ij}+w_jZ_i-w_iZ_j-\left|\begin{matrix}v_i-w_i &\d_{w_i}\left(\ln\frac{f^n(t,w)}{\ka}\right)_+\\  \\ v_j-w_j&\d_{w_j}\left(\ln\frac{f^n(t,w)}{\ka}\right)_+\end{matrix}\right|
\ea
$$
and
$$
Z_i=\d_{v_i}\left(\ln\frac{f^n(t,v)}{\ka}\right)_+\,,\quad Z_j=\d_{v_j}\left(\ln\frac{f^n(t,v)}{\ka}\right)_+\,,\quad Z_{ij}=v_iZ_j-v_jZ_i\,.
$$
Hence
$$
\ba
\int_{\bR^3}e^{-\l|w|^2}f^n(t,w)\left(\begin{matrix}1\\ w_j\\ -w_i\end{matrix}\right)q_{ij}^{f^n}(t,v,w)dw=\Ga_{ji}[\l,f^n(t,\cdot)]\left(\begin{matrix}Z_{ij}\\ Z_i\\ Z_j\end{matrix}\right)
\\
-\int_{\bR^3}e^{-\l|w|^2}\left(\begin{matrix}1\\ w_j\\ -w_i\end{matrix}\right)\left|\begin{matrix}v_i-w_i &\d_{w_i}(f^n(t,w)-\ka)_+\\  \\ v_j-w_j&\d_{w_j}(f^n(t,w)-\ka)_+\end{matrix}\right|dw\,,
\ea
$$
since $f^n\d_{w_k}(\ln f^n/\ka)_+=\d_{w_k}f^n\indc_{f^n>\ka}=\d_{w_k}(f^n-\ka)_+$, or equivalently
$$
\int_{\bR^3}e^{-\l|w|^2}f^n(t,w)\left(\begin{matrix}1\\ w_j\\ -w_i\end{matrix}\right)q_{ij}^{f^n}(t,v,w)dw+S(t,v)=\Ga_{ji}[\l,f^n(t,\cdot)]\left(\begin{matrix}Z_{ij}\\ Z_i\\ Z_j\end{matrix}\right)\,.
$$
We have denoted
$$
\Ga_{ji}[\l,f^n(t,\cdot)]:=\int_{\bR^3}e^{-\l|w|^2}f^n(t,w)\left(\begin{matrix}1 &w_j &-w_i\\ w_j &w_j^2 &-w_jw_i\\ -w_i &-w_iw_j &w_i^2\end{matrix}\right)dw
$$
and
$$
\ba
S(t,v):=&\int_{\bR^3}e^{-\l|w|^2}\!\left(\begin{matrix}1\\ w_j\\ -w_i\end{matrix}\right)\left|\begin{matrix}\d_{w_j}&(v_j-w_j)\\ \\ \d_{w_i}&(v_i-w_i)\end{matrix}\right|(f^n(t,w)-\ka)_+dw
\\
=&\int_{\bR^3}(f^n(t,w)-\ka)_+\left|\begin{matrix}(v_j-w_j)&\d_{w_j}\\ \\ (v_i-w_i)&\d_{w_i}\end{matrix}\right|e^{-\l|w|^2}\!\left(\begin{matrix}1\\ w_j\\ -w_i\end{matrix}\right)dw
\\
=&\int_{\bR^3}(f^n(t,w)-\ka)_+e^{-\l|w|^2}\left(\begin{matrix}2\l(v_iw_j-v_jw_i)\\ 2\l(v_iw_j-v_jw_i)w_j+w_i-v_i\\ 2\l(v_jw_i-v_iw_j)w_i+w_j-v_j\end{matrix}\right)dw\,.
\ea
$$
By Cramer's formula
$$
\ba
\Det\Ga_{ji}[\l,f^n(t,\cdot)]Z_i
\\
=\Det\int e^{-\l|w|^2}f^n(t,w)\left[\begin{matrix}1 &q_{ij}^{f^n}(t,v,w) &-w_i\\ w_j &w_jq_{ij}^{f^n}(t,v,w) &-w_jw_i\\ -w_i &-w_iq_{ij}^{f^n}(t,v,w) &w_i^2\end{matrix}\right]dw
\\
+\Det\!\!\!\int\!\!\!e^{-\l|w|^2}\!\left[\begin{matrix}f^n(t,w)\! &\!\!\!2\l v_i\wedge w_j(f^n(t,w)\!-\!\ka)_+\! &\!\!\!-w_if^n(t,w)
\\ w_j f^n(\!t,w\!)&\!\!\!(\!2\l w_j(\!v_i\!\wedge\! w_j\!)\!+\!w_i\!-\!v_i\!)(\!f^n(\!t,w\!)\!\!-\!\!\ka\!)_+ &\!\!\!-w_jw_if^n(\!t,w\!)\\ 
-w_if^n(\!t,w\!) &\!\!\!(\!2\l w_i(\!v_j\!\wedge\!w_i\!)\!+\!w_j\!-\!v_j\!)(\!f^n(\!t,w\!)\!\!-\!\!\ka\!)_+&\!\!\!w_i^2f^n(\!t,w\!)\end{matrix}\right]\!dw
\ea
$$
with the notation $a_i\wedge b_j:=a_ib_j-a_jb_i=-b_j\wedge a_i$. Using the lower bound for $\Det\Ga_{ji}[\l,f^n(t,\cdot)]$ obtained in Lemma 2 of \cite{DesviJFA}, one finds that
$$
\ba
|Z_i|\le &6C_D(M_2(t,f^n)^2,H(f_{in}))M_2(t,f^n)^2
\\
&\times\Bigg(\int_{\bR^3}f^n(t,w)q_{ij}^{f^n}(t,v,w)(1+|w|)e^{-\l|w|^2}dw
\\
&+\left(1+\tfrac1{\sqrt{2e\l}}\right)\max(1,2\l)(1+|v|)\int_{\bR^3}(f^n(t,w)-\ka)_+dw\Bigg)\,.
\ea
$$
Therefore
$$
\ba
4\int_{\bR^3}&\frac{|\d_{v_i}\sqrt{f^n(t,v)}|^2}{(1+|v|)^3}\indc_{f^n(t,v)>\ka}dv=\int_{\bR^3}\frac{f^n(t,v)|Z_i|^2}{(1+|v|)^3}dv
\\
\le&72C_D(M_2(t,f^n)^2,H(f_{in}))^2M_2(t,f^n)^4
\\
&\times\Bigg(\left(1+\tfrac1{\sqrt{e\l}}\right)^2\int_{\bR^3}\frac{f^n(t,v)dv}{(1+|v|)^3}\left(\int_{\bR^3}f^n(t,w)q_{ij}^{f^n}(t,v,w)e^{-\l|w|^2/2}dw\right)^2
\\
&+\left(1+\tfrac1{\sqrt{2e\l}}\right)^2\max(1,4\l^2)\int_{\bR^3}\frac{f^n(t,v)dv}{1+|v|}\left(\int_{\bR^3}(f^n(t,w)-\ka)_+dw\right)^2\Bigg)
\\
\le&72C_D(M_2(t,f^n)^2,H(f_{in}))^2M_2(t,f^n)^4\left(1+\tfrac1{\sqrt{e\l}}\right)^2\int_{\bR^3}f^n_{in}(v)dv
\\
&\times\Bigg(\iint_{\bR^3\times\bR^3}\frac{|v-w|^3e^{-\l|w|^2}}{(1+|v|)^3}\frac{|q_{ij}^{f^n}(t,v,w)|^2}{|v-w|^3}f^n(t,v)f^n(t,w)dvdw
\\
&+\max(1,4\l^2)\int_{\bR^3}f^n_{in}(v)dv\int_{\bR^3}(f^n(t,w)-\ka)_+dw\Bigg)\,.
\ea
$$
Since
$$
\frac{|v-w|^3e^{-\l|w|^2}}{(1+|v|)^3}\le4\frac{(|v|^3+|w|^3)e^{-\l|w|^2}}{(1+|v|)^3}\le4\frac{|v|^3+(3/2e\l)^{3/2}}{(1+|v|)^3}\le4\max(1,\tfrac3{2e\l})^{3/2}\,,
$$
we conclude that
$$
\ba
\int_{\bR^3}\frac{|\d_{v_i}\sqrt{f^n(t,v)}|^2}{(1+|v|)^3}\indc_{f^n(t,v)>\ka}dv
\\
\le18C_D(M_2(t,f^n)^2,H(f_{in}))^2M_2(t,f^n)^4M_0(f^n)\left(1+\tfrac1{\sqrt{e\l}}\right)^2
\\
\times\Bigg(4\max(1,\tfrac3{2e\l})^{3/2}\iint_{\bR^3\times\bR^3}\frac{|q_{ij}^{f^n}(t,v,w)|^2}{|v-w|^3}f^n(t,v)f^n(t,w)dvdw
\\
+\max(1,4\l^2)\int_{\bR^3}f^n_{in}(v)dv\int_{\bR^3}(f^n(t,w)-\ka)_+dw\Bigg)&\,.
\ea
$$
This concludes the proof of Lemma \ref{L-CoercH+}.


\section{Proof of Proposition \ref{P-DG}}\lb{S-DG}


\subsection{Replacing the localized entropy with Lebesgue norms}


\begin{Lem} Set $\mu(r):=\min(r,r^2)$ for each $r\ge 0$. Let $f$ be a suitable solution of the Landau equation, for some $q\in(1,2)$ and some negligible $\cN\subset\bR_+$. Denote
$$
f^+_\ka(t,v):=\mu((f(t,v)^{1/q}-\ka^{1/q})_+)\quad\text{ for all }\ka\ge 1\,.
$$
Then, for each $\io>0$, there exists $C(q,\io)>0$ such that, for all $\ka\in[1,2]$ and all $t_1<t_2\in\bR_+\setminus\cN$
\be\lb{Ineq2SuitSol}
\ba
\tfrac12c_h\int_{\bR^3}f^+_\ka(t_2,v)^qdv+\tfrac14C'_E\int_{t_1}^{t_2}\left(\int_{\bR^3}|\grad_vf^+_\ka(t,v)|^qdv\right)^{2/q}dt
\\
\le C(q,\io)\int_{\bR^3}f^+_\ka(t_1,v)^{q(1+\io)}+\ka^{1+\io}\indc_{f^+_\ka(t_1,v)>0})dv
\\
+2^q\ka\int_{t_1}^{t_2}\int_{\bR^3}\left(f^+_\ka(t,v)^q+2\ka\indc_{f^+_\ka(t,v)>0})\right)dvdt&\,.
\ea
\ee
\end{Lem}

\begin{proof} There exists $c_h>0$, and for each $\io>0$, there exists $C_\io>0$ such that
\be\lb{Mu<h}
c_h\mu((r-1)_+)\le h_+(r)\le C_\io(r-1)_+^{1+\io}\,,\quad\text{ for all }r\ge 0\,.
\ee
Thus, for each $\ka\in[1,2]$, since the functions $\mu$ and $z\mapsto(r^{1/q}-\ka^{1/q}+z)^q-z^q$ for $r>\ka$ are nondecreasing on $\bR_+$, one has
$$
\tfrac12c_h\mu((r^{1/q}-\ka^{1/q})^q_+)\le\tfrac12c_h\mu((r-\ka)_+)\le c_h\ka\mu((r/\ka -1)_+)\le c_h\ka h_+(r/\ka-1)\,,
$$
and since $\mu(r^q)=\mu(r)^q$,
\be\lb{IntMu<H+}
\tfrac12c_h\int_{\bR^3}\mu((f(t,v)^{1/q}-\ka^{1/q})_+)^qdv\le H_+(f(t,\cdot)|\ka)\,.
\ee
On the other hand, by convexity of $z\mapsto z^q$ on $\bR_+^*$
$$
(r-\ka)_+\le 2^{q-1}(r^{1/q}-\ka^{1/q})_+^q+(2^{q-1}-1)\ka\indc_{r>\ka}\le 2^{q-1}(\mu((r^{1/q}-\ka^{1/q})_+)^q+2\ka\indc_{r>\ka})\,,
$$
where the second inequality follows from observing that $y\indc_{f>\ka}=\mu(y\indc_{f>\ka})$ if $y\ge 1$, while $y\indc_{f>\ka}\le\indc_{f>\ka}\le\ka\indc_{f>\ka}$ otherwise. Hence
\be\lb{H+<mu}
\ba
H_+(f(t,\cdot)|\ka)\le C_\io\ka^{-\io}\int_{\bR^3}(f(t,v)-\ka)_+^{1+\io}dv
\\
\le 2^{(q-1)(1+\io)}C_\io\int_{\bR^3}(\mu((f(t,v)^{1/q}-\ka^{1/q})_+)^q+2\ka\indc_{f(t,v)>\ka})^{1+\io}dv
\\
\le C(q,\io)\int_{\bR^3}(\mu((f(t,v)^{1/q}-\ka^{1/q})_+)^{q(1+\io)}+\ka^{1+\io}\indc_{f(t,v)>\ka})dv&\,,
\ea
\ee
with $C(q,\io):=2^{(1+\io)(\io+q-1)}C_\io$. Finally, 
\be\lb{IntGradMu<}
\int_{\bR^3}|\grad_v(\mu((f(t,v)^{1/q}-\ka^{1/q})_+)|^qdv\le 2^q\int_{\bR^3}|\grad_v((f(t,v)^{1/q}-\ka^{1/q})_+)|^qdv\,,
\ee
since $0\le\mu'(r)\le 2$ for all $r\ge 0$.
\end{proof}


\subsection{The De Giorgi Method}


For each integer $k\ge 0$, set
$$
t^k:=\tfrac12-\tfrac14\cdot 2^{-k}\,,\quad\text{ and }\quad\ka_k:=(1+(2^{1/q}-1)(1-2^{-k}))^q\,.
$$
In this proof, we abuse the notation and write for simplicity
$$
f^+_k:=f^+_{\ka_k}\,.
$$
Set
$$
A_k:=\Esssup_{t^k\le t\le 1}\tfrac{c_h}2\int_{\bR^3}f^+_k(t,v)^qdv+\tfrac14{C'_E}\int_{t^k}^{1}\left(\int_{\bR^3}|\grad_vf^+_k(t,v)|^qdv\right)^{2/q}dt\,.
$$

\subsubsection{Step 1: the recurrence inequality}

Writing \eqref{Ineq2SuitSol} for $t_1:=\th\in[t^k,t^{k+1}]\setminus\cN$ and each $t_2\in[t^{k+1},1]\setminus\cN$ and letting $t_2\to 1^-$ in the integrals over $[t_1,t_2]$ implies that
$$
\ba
A_{k+1}\le 2\max\left(\Esssup_{t^{k+1}\le t\le 1}\tfrac{c_h}2\int_{\bR^3}f^+_{k+1}(t,v)^qdv,\tfrac{C'_E}4\int_{t^{k+1}}^{1}\left(\int_{\bR^3}|\grad_vf^+_{k+1}(t,v)|^qdv\right)^{2/q}dt\right)
\\
\le 2C(q,\io)\int_{\bR^3}f^+_{k+1}(\th,v)^{q(1+\io)}+\ka_{k+1}^{1+\io}\indc_{f^+_{k+1}(\th,v)>0})dv
\\
+2^{q+1}\ka_{k+1}\int_{\th}^1\int_{\bR^3}\left(f^+_{k+1}(t,v)^q+2\ka_{k+1}\indc_{f^+_{k+1}(t,v)>0})\right)dvdt&\,.
\ea
$$
Averaging the left- and rightmost sides of this inequality in $\th\in[t^k,t^{k+1}]$ shows that
$$
A_{k+1}\!\!\le\! 2^{k+5}(C(q,\io)+\!2^{q+1})\!\!\!\int_{t^k}^1\!\!\!\int_{\bR^3}\!\left(f^+_{k+1}(\th,v)^{q(1+\io)}\!\!+\!\!f^+_{k+1}(\th,v)^q\!+\!2\ka_{k+1}^{1+\io}\indc_{f^+_{k+1}(\th,v)>0}\right)\!dvd\th
$$
since $t_k-t_{k+1}=2^{-k-4}$ and $\ka_{k+1}\le 2$. Here comes the nonlinearization part of DeGiorgi's argument. First $f^+_{k+1}\le f^+_k$, and
$$
f^+_{k+1}>0\implies f>\ka_{k+1}\implies f^+_k>\mu(\ka_{k+1}^{1/q}-\ka_k^{1/q})=\mu((2^{1/q}-1)\cdot 2^{-k-1})\ge\frac{4^{-k-1}}{(1+\sqrt{2})^2}\,,
$$
so that
$$
\indc_{f^+_{k+1}>0}\le\indc_{f^+_k>4^{-k-3}}\,,\quad\text{ and hence }\quad\indc_{f^+_{k+1}>0}\le 4^{(k+3)p_0}(f^+_k)^{p_0}\,.
$$
Thus
$$
\ba
(f^+_{k+1})^{q(1+\io)}+(f^+_{k+1})^q+\ka_{k+1}^{1+\io}\indc_{f^+_{k+1}>0}
\\
\le(f^+_k)^{q(1+\io)}+4^{-(k+3)q}(4^{k+3}f^+_k)^{q(1+\io)}+2^{1+\io}(4^{k+3}f^+_k)^{q(1+\io)}\,,
\ea
$$
so that
$$
A_{k+1}\!\!\le\! 2^{k+5}(1+4^{(k+3)q\io}+2^{1+\io}4^{(k+3)q(1+\io)})(C(q,\io)\!\!+\!2^{q+1})\!\!\int_{t^k}^1\!\!\int_{\bR^3}\!f^+_k(\th,v)^{q(1+\io)}dvd\th.
$$
By H\"older's inequality and Sobolev's embedding (Theorem IX.9 in \cite{Brezis})
$$
\ba
\|f^+_k\|_{L^r([t_1,t_2]\times\bR^3)}\le&\|f^+_k\|_{L^\infty(t_1,t_2;L^q(\bR^3))}^{1-\a}\|f^+_k\|_{L^2(t_1,t_2;L^{q^*}(\bR^3))}^\a
\\
\le&C_S(q,3)^\a\|f^+_k\|_{L^\infty(t_1,t_2;L^q(\bR^3))}^{1-\a}\|\grad_vf^+_k\|_{L^2(t_1,t_2;L^q(\bR^3))}^\a
\ea
$$
with
$$
\frac1r=\frac{\a}2=\frac{1-\a}q+\frac{\a}{q^*}\text{ and }\frac1{q^*}=\frac1q-\frac13\,,\quad\text{ so that }r=\tfrac53q\,,\quad\text{ and }\quad\a=\frac{6}{5q}\,.
$$
Notice that we use the restriction $q>\tfrac65$ at this point. With $\io=\tfrac23$, one finds
$$
\ba
\int_{t^k}^1\!\!\int_{\bR^3}\!f^+_{k+1}(\th,v)^{\frac{5q}{3}}dvd\th\le &C_S(q,3)^2\|f^+_k\|_{L^\infty(t_k,1;L^q(\bR^3))}^{\frac{5q}{3}-2}\|\grad_vf^+_k\|_{L^2(t_k,1;L^q(\bR^3))}^2
\\
\le &C_S(q,3)^2(\tfrac2{c_h})^{\frac53-\frac2q}\tfrac{4}{C_E}A_k^{\frac83-\frac2q}\,,
\ea
$$
so that
$$
A_{k+1}\le M\L^kA_k^\b\,,\qquad k\ge 0\,,
$$
with 
$$
\b:=\frac83-\frac2q>1\quad\text{ and }\quad\L:=2\cdot 4^{5q/3}\,,
$$
while
$$
M:=2^{5+10q}\cdot 3\cdot (C(q,\tfrac23)\!\!+\!2^{q+1})C_S(q,3)^2\left(\frac2{c_h}\right)^{\frac53-\frac2q}\frac{4}{C_E}\,.
$$
An easy induction shows that, for all $k\ge 1$, one has
$$
A_k\le M^{-\frac1{\b-1}}\L^{-\frac1{(\b-1)^2}}\L^{\frac{k}{\b-1}}\left(M^{\frac1{\b-1}}\L^{\frac1{(\b-1)^2}}A_0\right)^{\b^k}\,.
$$
Hence
\be\lb{CondA0}
A_0<M^{-\frac1{\b-1}}\L^{-\frac1{(\b-1)^2}}\implies A_k\to 0\text{ as }k\to+\infty\,,
\ee
and by Fatou's lemma
$$
\int_{1/2}^1\int_{\bR^3}\mu((f(t,v)^{1/q}-2^{1/q})_+)dv\le\tfrac2{c_h}\varliminf_{k\to\infty}A_k=0\,,
$$ 
which implies in turn that $f(t,v)=0$ for a.e. $(t,v)\in[\tfrac12,1]\times\bR^3$.

\subsubsection{Step 2: initialization}

Write \eqref{SuitSolIneq} for $t_1\in[0,\tfrac14]\setminus\cN$ and $t_2\in[\tfrac14,1]\setminus\cN$, with $\ka=1$. Letting $t_2\to 1^-$ in the integrals over $[t_1,t_2]$ and using \eqref{IntMu<H+}-\eqref{IntGradMu<} shows that
$$
\ba
A_0=&\Esssup_{\frac14\le t\le 1}\tfrac{c_h}{2}\int_{\bR^3}\mu((f(t,v)^{1/q}\!-\!1)_+)^qdv
\\
&+\tfrac{C'_E}4\int_{1/4}^{1}\left(\int_{\bR^3}|\grad_v\mu((f(t,v)^{1/q}\!-\!1)_+)|^qdv\right)^{2/q}dt
\\
\le&2\left(H_+(f(t_1,\cdot)|1)+2\int_{t_1}^1\int_{\bR^3}(f(t,v)-1)_+dvdt\right)\,.
\ea
$$
Averaging in $t_1$ over $[\tfrac18,\tfrac14]$ shows that
$$
A_0\le 16\left(\int_{1/8}^1H_+(f(t_1,\cdot)|1)dt_1+2\int_{1/8}^1\int_{\bR^3}(f(t,v)-1)_+dvdt\right)\,.
$$
Since $(f-1)_+\le(f-\tfrac12)_+$ and
$$
(f-1)_+\le(f-1)_+\indc_{(f-\frac12)_+\ge\frac12}\le 2(f-1)_+(f-\tfrac12)_+\le 2(f-\tfrac12)^2_+\,,
$$
one has, by \eqref{Mu<h},
$$
(f-1)_+\le 2\mu((f-\tfrac12)_+)\le\tfrac1{c_h}h_+(2f)\,.
$$
Hence
$$
\ba
A_0\le 16\left(\int_{1/8}^1H_+(f(t_1,\cdot)|1)dt_1+\tfrac1{c_h}\int_{1/8}^1H_+(f(t,\cdot)|\tfrac12)dt\right)
\\
\le 16(1+\tfrac1{c_h})\int_{1/8}^1H_+(f(t,\cdot)|\tfrac12)dt&\,.
\ea
$$
Together with \eqref{CondA0}, this implies Proposition \ref{P-DG} with
$$
\eta_0:=\frac{c_h^2}{1+c_h}\left(\frac{C'_E}{3(C(q,\tfrac23)\!+\!2^{q+1})C_S(q,3)^2}\right)^{\frac1{\b-1}}2^{-5-\frac{7+12q}{\b-1}-\frac4{(\b-1)^2}}\,.
$$


\section{Proof of Proposition \ref{P-ImprovedDG}}\lb{S-ImprovedDG}


The key idea in the proof of Proposition \ref{P-ImprovedDG} is to apply the De Giorgi local boundedness argument (Proposition \ref{P-DG}) to a scaled suitable solution of \eqref{LC}. As explained in section \ref{S-MR}, the 2-parameter group
of scaling transformations \eqref{ScalLC} leaving \eqref{LC} invariant is not rich enough to contain a subgroup leaving the conserved quantities \eqref{ConsLaw} and the entropy invariant. However, this is unessential for our argument, and 
we only seek to leave the truncated entropy inequality \eqref{SuitSolIneq} invariant. 

\subsection*{Step 1: Scaling solutions to the Landau equation with Coulomb interaction}

Let $f$ be a suitable solution to the Landau equation with Coulomb interaction, satisfying \eqref{SuitSolIneq} for some Lebesgue negligible set $\cN\subset(0,+\infty)$, some $q\in(\tfrac43,2)$, and with a given constant $C'_E>0$. According
to Proposition \ref{P-ExistSSol}, the existence of such a solution is known provided that the initial data $f\rstr_{t=0}=:f_{in}$ satisfies \eqref{Momfin} with $q=\frac{2k}{k+3}$ with $k>6$.

For each integer $n\ge 0$, set $\eps_n:=2^{-n}$ and
$$
f_n(t,v):=\eps_n^\g f(1+\eps_n^\g(t-1),\eps_nv)\,,\qquad\text{ with }\g:=\frac{5q-6}{2q-2}\,.
$$
(Observe that, up to the translation in time, this is precisely the subgroup of scaling transformations \eqref{ScalLCq}). One easily checks that $f_n$ is a H-solution to the Landau equation for each $n\ge 0$, that
$$
\int_{\bR^3}f_n(t,v)dv=\eps_n^{\g-3}\int_{\bR^3}f(1+\eps_n^\g(t-1),V)dV=\int_{\bR^3}f_{in}(V)dV\,,
$$
while
\be\lb{ScalH+}
\ba
H_+(f_n(t,\cdot)|\eps_n^\g\ka)=\eps_n^{\g-3}H_+(f(1+\eps_n^\g(t-1),\cdot)|\kappa)
\\
\int_{t_1}^{t_2}\int_{\bR^3}(f_n(t,v)-\eps_n^\g\ka)_+dvdt=\frac1{\eps_n^\g}\int_{1+\eps^\g_n(t_1-1)}^{1+\eps^3_n(t_2-1)}\int_{\bR^3}(f(T,V)-\ka)_+dVdT\,,
\ea
\ee
and that
\be\lb{ScalTruncProd}
\ba
\int_{t_1}^{t_2}\left(\int_{\bR^3}|\grad_v(f_n(t,v)^{1/q}-(\eps_n^\g\ka)^{1/q}|^qdv\right)^{2/q}dt
\\
=\eps_n^{\g-3}\int_{1+\eps_n^\g(t_1-1)}^{1+\eps_n^\g(t_2-1)}\left(\int_{\bR^3}|\grad_V(f(T,V)^{1/q}-\ka^{1/q}|^qdV\right)^{2/q}dT&\,.
\ea
\ee
Since $f$ satisfies \eqref{SuitSolIneq} for all $\ka\ge 1$ and all $t\in\cN$, the scaling transformations \eqref{ScalH+} and \eqref{ScalTruncProd} imply that
\be\lb{TruncEntrIneq-n}
\ba
H_+(f_n(t_2,\cdot)|\eps_n^\g\ka)+C'_E\int_{t_1}^{t_2}\left(\int_{\bR^3}|\grad_v(f_n(t,v)^{1/q}-(\eps_n^\g\ka)^{1/q}|^qdv\right)^{2/q}dt
\\
\le H_+(f_n(t_1,\cdot)|\eps_n^\g\ka)+2\eps_n^\g\ka\int_{t_1}^{t_2}\int_{\bR^3}(f_n(t,v)-\eps_n^\g\ka)_+dvdt&\,,
\ea
\ee
for all $\ka\ge 1$ and all $t\in\cN_n$, where
\be\lb{N-n}
\cN_n:=\{t\ge 0\hbox{ s.t. }1+\eps_n^\g(t-1)\in\cN\}\,.
\ee
Since $\cN_n$ is the image of the Lebesgue negligible set $\cN$ by an affine transformation, it is Lebesgue negligible. On the other hand $f_n$ satisfies the truncated entropy inequality \eqref{TruncEntrIneq-n} for all $\kappa\ge 1$, or 
equivalently whenever $\ka_n:=\eps_n^\g\ka\ge 2^{-n\g}$. Since $q>\tfrac43$, one has $\g>0$, so that $[1,+\infty)\subset[2^{-n\g},\infty)$. 

This proves that $f_n$ is a suitable solution to the Landau equation satisfying \eqref{SuitSolIneq} with the same Lebesgue exponent $q\in(\tfrac43,2)$, the same constant $C'_E$, and with the new Lebesgue-negligible set $\cN_n$.

From now on, and until the very end of this proof, we forget completely the suitable solution $f$, and consider only the sequence of scaled solutions $f_n$. The goal is to show that, by choosing $\eta_1$ small enough, one can find some 
$n$ large enough so that the scaled solution $f_n$ satisfies the assumption of Proposition \ref{P-DG}. Applying Proposition \ref{P-DG} shows that this $f_n$ is locally bounded near $t=1$, which implies in turn that $f$ is locally bounded
(by a very large number, whose size is unessential) near $t=1$ (here again, the size of the domain on which $f$ is bounded is of no interest for the partial regularity result).

\subsubsection*{Step 2: Replacing the local entropy with Lebesgue norms}

In this section we seek to apply \eqref{Ineq2SuitSol} to each suitable solution $f_n$ with truncation parameter $\ka_n=\eps_n^\g\ka=1$. 

Thus we define $F_n(t,v):=\mu((f_n(t,v)^{1/q}-1)_+)$ for each integer $n\ge 0$, and for a.e. $(t,v)\in[0,+\infty)\times\bR^3$. Then
\be\lb{IntFn<1}
\ba
\int_{\bR^3}F_n(t,v)^qdv\le&\int_{\bR^3}(f_n(t,v)^{1/q}-1)_+^qdv
\\
\le&\int_{\bR^3}f_n(t,v)dv=\eps_n^{\g-3}\int_{\bR^3}f_{in}(v)dv=\eps_n^{\g-3}\,,
\ea
\ee
while
$$
\ba
\int_0^1\left(\int_{\bR^3}|\grad_vF_n(t,v)|^qdv\right)^{2/q}dt
\\
=\int_0^1\left(\int_{\bR^3}|\mu'((f_n(t,v)^{1/q}-1)_+)\grad_v(f_n(t,v)^{1/q}-1)_+|^qdv\right)^{2/q}dt
\\
\le 4\int_0^1\left(\int_{\bR^3}|\grad_v(f_n(t,v)^{1/q}-1)_+|^qdv\right)^{2/q}dt
\\
=4\eps_n^{\g-3}\int_{1-\eps_n^\g}^1\left(\int_{\bR^3}|\grad_V(f(T,V)^{1/q}-\eps_n^{-g/q})_+|^qdV\right)^{2/q}dT<8\eta_1\,,
\ea
$$
for all $n\ge N$ large enough. 

Writing \eqref{Ineq2SuitSol} for each $f_n$, replacing the truncation parameter $\kappa$ in \eqref{Ineq2SuitSol} with $\ka_n=1$, shows that, for $t_1<t_2\in[0,+\infty)\setminus\cN_n$ and all $\io>0$, one has
$$
\ba
\tfrac{c_h}2\int_{\bR^3}F_{n+1}(t_2,v)^qdv+\tfrac{C'_E}{4}\int_{t_1}^{t_2}\left(\int_{\bR^3}|\grad_vF_{n+1}(t,v)|^qdv\right)^{2/q}dt
\\
\le C(q,\io)\int_{\bR^3}\left(F_{n+1}(t_1,v)^{q(1+\io)}+\indc_{F_{n+1}(t_1,v)>0}\right)dv
\\
+2^q\int_{t_1}^{t_2}\int_{\bR^3}\left(F_{n+1}(t,v)^q+2\indc_{F_{n+1}(t,v)>0}\right)dvdt&\,.
\ea
$$
Taking the $\Esssup$ of the left hand side for $\tfrac12<t_2<1$ and averaging the right hand side in $t_1\in[0,\tfrac12]$ shows that
\be\lb{supFn+1<}
\ba
\Esssup_{\frac12<t_2<1}\int_{\bR^3}F_{n+1}(t_2,v)^qdv
\\
\le\frac{4C(q,\io)}{c_h}\int_0^{1/2}\int_{\bR^3}\left(F_{n+1}(t_1,v)^{q(1+\io)}+\indc_{F_{n+1}(t_1,v)>0}\right)dvdt_1
\\
+\frac{2^{q+2}}{c_h}\int_0^1\int_{\bR^3}\left(F_{n+1}(t,v)^q+2\indc_{F_{n+1}(t,v)>0}\right)dvdt
\\
\le\tfrac{4C(q,\io)+2^{q+3}}{c_h}\int_0^1\int_{\bR^3}\left(F_{n+1}(t,v)^{q(1+\io)}+F_{n+1}(t,v)^q+\indc_{F_{n+1}(t,v)>0}\right)dvdt&\,.
\ea
\ee

\subsection*{Step 3: Using the Sobolev embedding}

We recall that 
$$
\|\phi\|_{L^{q^*}(\bR^3)}\le C_S(q,3)\|\grad\phi\|_{L^q(\bR^3)}\quad\text{ with }\quad q^*=\frac{3q}{3-q}
$$ 
for all $q\in[1,3)$ and all $\phi\in W^{1,q}(\bR^3)$ (see Theorem IX.9 in \cite{Brezis}). Thus, for each $\io\ge 0$,
$$
\ba
\int_0^1\int_{\bR^3}F_{n+1}(t,v)^{q(1+\io)}dvdt
\\
\le\|F_{n+1}^{q(1+\io)}\|_{L^\frac{2}{q(1+\io)}(0,1;L^\frac{q^*}{q(1+\io)}(\bR^3))}\|\indc_{F_{n+1}>0}\|_{L^\frac{2}{2-q(1+\io)}(0,1;L^\frac{3}{3-(3-q)(1+\io)}(\bR^3))}
\\
\le\|F_{n+1}\|^{q(1+\io)}_{L^2(0,1;L^{q^*}(\bR^3))}\Esssup_{0<t<1}|\{v\in\bR^3\text{ s.t. }F_{n+1}(t,v)>0\}|^\frac{q-(3-q)\io}{3}
\\
\le C_S(q,3)^{q(1+\io)}\|\grad_vF_{n+1}\|^{q(1+\io)}_{L^2(0,1;L^q(\bR^3))}\Esssup_{0<t<1}|\{v\in\bR^3\text{ s.t. }F_{n+1}(t,v)>0\}|^\frac{q-(3-q)\io}{3}
\\
\le (8C_S(q,3))^{q(1+\io)}\eta_1^\frac{q(1+\io)}2\Esssup_{0<t<1}|\{v\in\bR^3\text{ s.t. }F_{n+1}(t,v)>0\}|^\frac{q-(3-q)\io}{3}&\,.
\ea
$$
Observe that
$$
\ba
F_{n+1}(t,v)>0\implies f_{n+1}(t,v)>1
\\
\implies(f_{n}(1+2^{-\g}(t-1),v)^{1/q}-1)_+>2^{\g/q}-1
\\
\implies F_n(1+2^{-\g}(t-1),\tfrac12v)^q>\mu(2^{\g/q}-1)^q&\,.
\ea
$$
Hence
$$
|\{v\in\bR^3\text{ s.t. }F_{n+1}(t,v)>0\}|\le\frac{8}{\mu(2^{\g/q}-1)^q}\int_{\bR^3}F_n(1+2^{-\g}(t-1),V)^qdV\,,
$$
so that
\be\lb{Fn+1qio<}
\ba
\int_0^1\int_{\bR^3}F_{n+1}(t,v)^{q(1+\io)}dvdt
\\
\le(8C_S(q,3))^{q(1+\io)}\left(\tfrac{8}{\mu(2^{\g/q}-1)^q}\right)^\frac{q-(3-q)\io}3\eta_1^\frac{q(1+\io)}2\Esssup_{1-2^{-\g}<T<1}\left(\int_{\bR^3}F_n(T,V)^qdV\right)^\frac{q-(3-q)\io}3&\,.
\ea
\ee
Likewise, setting $\io=0$ in the inequality above
\be\lb{Fn+1q<}
\ba
\int_0^1\int_{\bR^3}F_{n+1}(t,v)^qdvdt
\\
\le(8C_S(q,3))^q\left(\tfrac{8}{\mu(2^{\g/q}-1)^q}\right)^\frac{q}3\eta_1^\frac{q}2\Esssup_{1-2^{-\g}<T<1}\left(\int_{\bR^3}F_n(T,V)^qdV\right)^\frac{q}3&\,.
\ea
\ee
Finally
$$
\indc_{F_{n+1}(t,v)>0}\le\frac{F_n(1+2^{-\g}(t-1),\tfrac12v)^q}{\mu(2^{\g/q}-1)^q}\indc_{F_{n+1}(t,v)>0}
$$
so that
\be\lb{1Fn+1<}
\ba
\int_0^1\int_{\bR^3}\indc_{F_{n+1}(t,v)>0}dvdt\le\frac{2^{\g+3}}{\mu(2^{\g/q}-1)^q}\int_{7/8}^1\int_{\bR^3}F_n(T,V)^qdVdT
\\
\le\frac{2^{\g+3+4q}C_S(q,3)^q}{\mu(2^{\g/q}-1)^{q+\frac{q^2}3}}\eta_1^\frac{q}2\Esssup_{1-2^{-\g}<T<1}\left(\int_{\bR^3}F_{n-1}(T,V)^qdV\right)^\frac{q}3&\,.
\ea
\ee

Henceforth, we assume that $q\in(\tfrac43,2)$, so that $\g>1$, which implies in turn that $1-2^{-\g}>\tfrac12$. Putting together \eqref{supFn+1<}, \eqref{Fn+1qio<}, \eqref{Fn+1q<}, and \eqref{1Fn+1<} shows that
\be\lb{IterDGImpr}
\ba
\Esssup_{\frac12<t_2<1}\int_{\bR^3}F_{n+1}(t_2,v)^qdv\le&D(q,\io)\eta_1^\frac{q}2\max\left(1,\Esssup_{\frac12<T<1}\int_{\bR^3}F_n(T,V)^qdV\right)^\frac{q}3
\\
&+D(q,\io)\eta_1^\frac{q}2\max\left(1,\Esssup_{\frac12<T<1}\int_{\bR^3}F_{n-1}(T,V)^qdV\right)^\frac{q}3&\,,
\ea
\ee
with
$$
D(q,\io):=\tfrac{4C(q,\io)+2^{q+3}}{c_h}\max\left(\tfrac{2^{(4q+\io)}C_S(q,3)^{q(1+\io)}}{\mu(2^{\g/q}-1)^{\frac{q^2(1+\io)}3-q\io}}
+\tfrac{2^{4q}C_S(q,3)^q}{\mu(2^{\g/q}-1)^{\frac{q^2}3}},\tfrac{2^{\g+3+4q}C_S(q,3)^q}{\mu(2^{\g/q}-1)^{q+\frac{q^2}3}}\right)\,.
$$

\subsection*{Step 4: The induction argument}

\begin{Lem} Let $X_n$ be a sequence of positive numbers such that 
\be\lb{2LevIter}
X_{n+1}<\rho(\max(1,X_n)^\a+\max(1,X_{n-1})^\a)\,,\qquad n\ge 1\,,
\ee
where $0<\rho<\tfrac12$ and $X_0,X_1\le M$ with $M\ge 1$. Then
$$
X_{2n}\text{ and }X_{2n+1}\le \max\left(2\rho,(2\rho)^\frac{1-\a^n}{1-\a}M^{\a^n}\right)\qquad n\ge 1\,.
$$
\end{Lem}

\begin{proof}
This is proved by an elementary induction argument. The desired conclusion holds for $n=0$, i.e.
$$
X_0\text{ and }X_1\le M=\max(2\rho,M)\,,\quad\text{ since }M\ge 1\text{ and }2\rho<1\,.
$$
Assume that
$$
X_{2n}\text{ and }X_{2n+1}\le\max(2\rho,(2\rho)^{1+\a+\ldots+\a^{n-1}}M^{\a^n})\,.
$$
If $(2\rho)^{1+\a+\ldots+\a^{n-1}}M^{\a^n}\le 1$, the assumption $2\rho<1$ implies that
$$
X_{2n+2}\le\rho(1^\a+1^\a)=2\rho<1\,,\text{ and }X_{2n+3}\le\rho(\max(1,2\rho)^\a+1^\a)=2\rho<1\,.
$$
If on the other hand $(2\rho)^{1+\a+\ldots+\a^{n-1}}M^{\a^n}>1$, then
$$
\ba
X_{2n+2}\le&\rho\left((2\rho)^{\a(1+\a+\ldots+\a^{n-1})}M^{\a^{n+1}}+(2\rho)^{\a(1+\a+\ldots+\a^{n-1})}M^{\a^{n+1}}\right)
\\
=&2\rho(2\rho)^{\a(1+\a+\ldots+\a^{n-1})}M^{\a^{n+1}}=(2\rho)^{1+\a+\ldots+\a^n}M^{\a^{n+1}}
\\
\le&\max(2\rho,(2\rho)^{1+\a+\ldots+\a^n}M^{\a^{n+1}})\,.
\ea
$$
In particular
$$
\max\left(2\rho,(2\rho)^{1+\a+\ldots+\a^n}M^{\a^{n+1}}\right)\le(2\rho)^{1+\a+\ldots+\a^{n-1}}M^{\a^n}
$$
since 
$$
2\rho<1<(2\rho)^{1+\a+\ldots+\a^{n-1}}M^{\a^n}\,.
$$
Hence
$$
\ba
X_{2n+3}\le&\rho\left((2\rho)^{\a(1+\a+\ldots+\a^{n-1})}M^{\a^{n+1}}+\max\left(1,2\rho,(2\rho)^{1+\a+\ldots+\a^n}M^{\a^{n+1}}\right)^\a\right)
\\
\le&\rho\left((2\rho)^{\a(1+\a+\ldots+\a^{n-1})}M^{\a^{n+1}}+\left((2\rho)^{1+\a+\ldots+\a^{n-1}}M^{\a^{n}}\right)^\a\right)
\\
=&2\rho(2\rho)^{\a(1+\a+\ldots+\a^{n-1})}M^{\a^{n+1}}\le\max\left(2\rho,(2\rho)^{1+\a+\ldots+\a^n}M^{\a^{n+1}}\right)\,.
\ea
$$
since $2\rho<1$ and $M^{\a^{n+1}(\a-1)}\le 1$.
\end{proof}

Henceforth we choose $\io=\tfrac23$, as in the proof of the first De Giorgi type lemma. Choose $\eta_1$ small enough so that
$$
0<\eta_1<(2D(q,\tfrac23))^{-2/q}\,,
$$
and apply the lemma above to the sequence
$$
X_n=\Esssup_{\frac12<t<1}\int_{\bR^3}F_{N+n}(t,v)^qdv\,,
$$
with
$$
\rho:=D(q,\tfrac23)\eta_1^\frac{q}2\,,\quad\a=q/3\,,\quad\text{ and }\quad M:=2^{(N+1)(3-\g)}\,.
$$
Because of \eqref{IterDGImpr}, this sequence $X_n$ satisfies \eqref{2LevIter} for all $n\ge 0$. Writing \eqref{IntFn<1} with $n=N$ and $n=N+1$ shows that 
$$
X_0\le 2^{-N(\g-3)}\le M\quad\text{ and }X_1\le 2^{-(N+1)(\g-3)}\le M
$$
because $3-\g=\frac{q}{2q-2}>1$ since $q\in(\tfrac43,2)$.

Since
$$
(2\rho)^{1+\a+\ldots+\a^n}M^{\a^{n+1}}\to (2\rho)^{\frac1{1-\a}}<2\rho\quad\text{ as }n\to+\infty\,,
$$
there exists $n_0$ such that $n\ge N+n_0$ implies that
$$
\Esssup_{\frac12<t<1}\int_{\bR^3}F_n(t,v)^qdv\le 2\rho=2D(q,\tfrac23)\eta_1^\frac{q}2<1\,.
$$

\subsection*{Step 5: Conclusion}

Using \eqref{H+<mu} shows that 
$$
\ba
\int_0^1H_+(f_{N+n_0+2}(t,\cdot)|1)dt
\\
\le C(q,\tfrac23)\int_0^1\int_{\bR^3}(F_{N+n_0+2}(t,v)^{5q/3}+\indc_{F_{N+n_0+2}(t,v)>0})dvdt&\,.
\ea
$$
At this point, we use \eqref{Fn+1qio<} and \eqref{1Fn+1<} to bound the right hand side of the inequality above, which leads to
$$
\ba
\int_0^1H_+(f_{N+n_0+2}(t,\cdot)|1)dt
\\
\le C(q,\tfrac23)\tfrac{2^{\frac{20q-6}3}}{C_S(q,3)^\frac{5q}3\mu(2^{\g/q}-1)^\frac{q(5q-6)}9}\eta_1^\frac{5q}6\Esssup_{1-2^{-\g}<T<1}\left(\int_{\bR^3}F_{N+n_0+1}(T,V)^qdV\right)^\frac{5q-6}9
\\
+C(q,\tfrac23)\tfrac{2^{\g+3+4q}C_S(q,3)^q}{\mu(2^{\g/q}-1)^{q+\frac{q^2}3}}\eta_1^\frac{q}2\Esssup_{1-2^{-\g}<T<1}\left(\int_{\bR^3}F_{N+n_0}(T,V)^qdV\right)^\frac{q}3
\\
\le C(q,\tfrac23)D(q,\tfrac23)\eta_1^\frac{q}2\max\left(1,\Esssup_{\frac12<T<1}\int_{\bR^3}F_{N+n_0+1}(T,V)^qdV\right)^\frac{q}3
\\
+ C(q,\tfrac23)D(q,\tfrac23)\eta_1^\frac{q}2\max\left(1,\Esssup_{\frac12<T<1}\int_{\bR^3}F_{N+n_0}(T,V)^qdV\right)^\frac{q}3
\\
\le 2C(q,\tfrac23)D(q,\tfrac23)\eta_1^\frac{q}2&\,.
\ea
$$
Using \eqref{ScalH+} shows that
$$
\ba
\int_0^1H_+(f_{N+n_0+3}(t,\cdot)|2^{-\g})dt=&\int_0^12^{3-\g}H_+(f_{N+n_0+2}(1+2^{-\g}(t-1),\cdot)|1)dt
\\
=&8\int_{1-2^{-\g}}^1H_+(f_{N+n_0+2}(T,\cdot)|1)dT
\\
\le& 16C(q,\tfrac23)D(q,\tfrac23)\eta_1^\frac{q}2\,.
\ea
$$
On the other hand, since $z\mapsto h_+(z):=(z\ln z-z+1)\indc_{z>1}$ is nondecreasing (observe for instance that $h'_+(r)=(\ln r)_+\ge 0$), one has
$$
\ka_1\le\ka_2\implies h_+(\phi/\ka_2)\le h_+(\phi/\ka_1)\implies \tfrac1{\ka_2}H_+(\phi|\ka_2)\le\tfrac1{\ka_1}H_+(\phi|\ka_1)
$$
for all $\phi\equiv\phi(v)$ measurable on $\bR^3$ and such that $\phi(v)\ge 0$ for a.e. $v\in\bR^3$. Since $q\in(\tfrac43,2)$, one has $\g>1$ and therefore
$$
\ba
\int_{1/8}^1H_+(f_{N+n_0+3}(t,\cdot)|\tfrac12)dt\le&\int_0^1H_+(f_{N+n_0+3}(t,\cdot)|\tfrac12)dt
\\
\le&2^{\g-1}\int_0^1H_+(f_{N+n_0+3}(t,\cdot)|2^{-\g})dt
\\
\le&2^{\g+3}C(q,\tfrac23)D(q,\tfrac23)\eta_1^\frac{q}2\,.
\ea
$$
Choosing $0<\eta_1$ small enough so that
$$
\eta_1<\eta_0[c_h,q,C'_E]^\frac2{q}(2^{\g+3}C(q,\tfrac23)D(q,\tfrac23))^\frac2{q}\,,
$$
we conclude that
$$
f_{N+n_0+3}(t,v)\le 2\qquad\text{ for a.e. }(t,v)\in[\tfrac12,1]\times\bR^3\,.
$$

\smallskip
At this point, we return to the original suitable solution $f$. The last inequality is equivalent to
$$
f(T,V)\le 2^{\g(N+n_0+3)+1}\qquad\text{ for a.e. }(t,v)\in[1-2^{-\g(N+n_0+3)-1},1]\times\bR^3\,,
$$
which completes the proof of Proposition \ref{P-ImprovedDG}.


\section{Proof of Theorem \ref{T-PartReg}}\lb{S-PartReg}


By Proposition \ref{P-ExistSSol}, there exist a negligible set $\cN\subset(0,+\infty)$, and, for each Lebesgue exponent $q\in(1,2)$, a constant $C'_E[T,f_{in},q]>0$ such that \eqref{SuitSolIneq} holds for all $t_1,t_2\in[0,T]\setminus\cN$ and 
all $\ka\ge 1$.

If $\tau\in\bS[f,[1,2]]$, consider the function $f_\tau:\,(t,v)\mapsto f(t+\tau-1,v)$. This is a suitable solution to the Landau equation \eqref{LC} on $[0,1]$ for which \eqref{SuitSolIneq} holds for each $q\in(1,2)$ with the constant $C'_E[1,f_{in},q]>0$,
for all $t_1,t_2\in[0,1]$ such that $t_1+\tau-1\notin\cN$ and $t_2+\tau-1\notin\cN$.

Applying Proposition \ref{P-ImprovedDG} to $f_\tau$ shows that, for each $q\in(\tfrac43,2)$, there exists $\eps(\tau)\in(0,\tfrac12)$ such that
$$
\int_{\tau-\eps(\tau)^\g}^\tau\left(\int_{\bR^3}|\grad_v(f(t,v)^{1/q}-\eps(\tau)^{-\g/q})_+|^q\right)^{2/q}dt\ge\tfrac12\eta_1\eps(\tau)^{3-\g}
$$
with $\g=\frac{5q-6}{2q-2}$. Observe that
$$
\grad_v(f(t,v)^{1/q}-\ka^{1/q})_+=\grad_v(f(t,v)^{1/q})\indc_{f(t,v)\ge\ka}\,,
$$
so that
$$
|\grad_v(f(t,v)^{1/q}-\eps(\tau)^{-\g/q})_+|\le|\grad_v(f(t,v)^{1/q}-1)_+|
$$
and therefore
$$
\int_{\tau-\eps(\tau)^\g}^\tau\left(\int_{\bR^3}|\grad_v(f(t,v)^{1/q}-1)_+|^q\right)^{2/q}dt\ge\tfrac12\eta_1\eps(\tau)^{3-\g}\,.
$$

Hence
$$
\bS[f,[1,2]]\subset\bigcup_{\tau\in\bS[f,[1,2]]}(\tau-\eps(\tau)^\g,\tau+\eps(\tau)^\g)\,.
$$
By the Vitali covering theorem (see chapter I, \S 1.6 in \cite{Stein}), there exists a countable subcollection of pairwise disjoint intervals $(\tau_j-\eps(\tau_j)^\g,\tau_j+\eps(\tau_j)^\g)$ such that
$$
\bS[f,[1,2]]\subset\bigcup_{j\ge 1}(\tau_j-5\eps(\tau_j)^\g,\tau_j+5\eps(\tau_j)^\g)\,.
$$
On the other hand
$$
\ba
\sum_{j\ge 1}\tfrac12\eta_1\eps(\tau)^{3-\g}\le&\sum_{j\ge 1}\int_{\tau_j-\eps(\tau_j)^\g}^{\tau_j}\left(\int_{\bR^3}|\grad_v(f(t,v)^{1/q}-1)_+|^q\right)^{2/q}dt
\\
=&\int_{\bigcup_{j\ge 1}(\tau_j-\eps(\tau_j)^\g,\tau_j)}\left(\int_{\bR^3}|\grad_v(f(t,v)^{1/q}-1)_+|^q\right)^{2/q}dt
\\
\le&\int_0^2\left(\int_{\bR^3}|\grad_v(f(t,v)^{1/q}-1)_+|^q\right)^{2/q}dt
\\
\le&\frac1{C'E}\left(H_+(f_{in}|1)+4\int_{\bR^3}f_{in}(v)dv\right)<\infty\,,
\ea
$$
where the equality above follows from the fact that the intervals $(\tau_j-\eps(\tau_j)^\g,\tau_j)$ are pairwise disjoint. Since $\g=\frac{5q-6}{2q-2}$, this proves that
$$
\cH^{\frac{q}{5q-6}}(\bS[f,[1,2]])<\infty\,,\quad\text{ for all }q\in(\tfrac43,2)\,.
$$
Since $\frac{q}{5q-6}$ decreases from $2$ to $\tfrac12$ as $q$ increases from $\tfrac43$ to $2$, we conclude that $\cH^s(\bS[f,[1,2]])<\infty$ for all $s>\tfrac12$, which implies in turn that $\cH^s(\bS[f,[1,2]])=0$ for all $s>\tfrac12$
(see for instance Theorem 2.1.3 in \cite{AmbroTilli}).

For each $m\in\bZ$, set $f_m(t,v):=2^{-m}f(2^{-m}t,v)$; then $f_m$ is a suitable solution to the Landau equation on $[0,T]\times\bR^3$ for each $T>0$, and
$$
\ba
\int_{\bR^3}(1+|v|^k+|\ln f_m(0,v)|)f_m(0,v)dv&
\\
=2^{-m}\int_{\bR^3}(1+|m|\ln 2+|v|^k+|\ln f_{in}(v)|)f_{in}(v)dv&<\infty\quad\text{ for all }k>3\,.
\ea
$$
Hence, for each $m\in\bZ$, one has
$$
\cH^s(\bS[f,[2^{-m},2^{1-m}]])=2^{-sm}\cH^s(\bS[f_m,[1,2]])=0\quad\text{ for all }s>\tfrac12\,.
$$
Therefore
$$
\cH^s(\bS[f,(0+\infty)])=\sum_{m\in\bZ}\cH^s(\bS[f,[2^{-m},2^{1-m}]])=0\quad\text{ for all }s>\tfrac12\,,
$$
which implies that $\bS[f,(0,+\infty)]$ has Hausdorff dimension $\le\tfrac12$ (see Definition 2.1.5 in \cite{AmbroTilli}).

\bigskip
\noindent
{\textbf{Acknowledgements.} We are most grateful to A.V. Bobylev, L. Desvillettes and Y. Martel for several helpful discussions during the preparation of this paper.


\end{document}